\documentclass[12 pt,reqno]{amsart}

\usepackage{latexsym, amsmath, bm, amssymb, longtable, booktabs,amscd,microtype,booktabs,cases}
\usepackage[square,numbers]{natbib}
\usepackage{graphicx}
\usepackage[dvipsnames]{xcolor}
\usepackage{caption}
\usepackage{dsfont}
\usepackage[all]{xy}
\usepackage{enumerate}
\usepackage{amsthm}
\usepackage{multirow}
\usepackage{tikz}
\usetikzlibrary{matrix,arrows,decorations.pathmorphing}
\usepackage{collectbox}
\usepackage{enumerate}
\usepackage[colorinlistoftodos,prependcaption, textsize=tiny]{todonotes}
\reversemarginpar
\usepackage{amsmath}
\usepackage{enumitem}
\usepackage{hyperref}
\let\oldtocsection=\tocsection

\let\oldtocsubsection=\tocsubsection

\let\oldtocsubsubsection=\tocsubsubsection

\renewcommand{\tocsection}[2]{\hspace{0em}{\vspace{0.5em}}\oldtocsection{#1}{#2}}
\renewcommand{\tocsubsection}[2]{\hspace{1em}{\vspace{0.5em}}\oldtocsubsection{#1}{#2}}
\renewcommand{\tocsubsubsection}[2]{\hspace{2em}\oldtocsubsubsection{#1}{#2}}
\hypersetup{
	colorlinks,
	citecolor=red,
	filecolor=black,
	linkcolor=blue,
	urlcolor=black
}

\numberwithin{equation}{section} 
\newcommand{\Z}{\mathbb{Z}}

\newcommand{\C}{\mathbb{C}}

\newcommand\N{\mathbb{N}}

\newcommand\bb{\bf}

\newcommand{\LL}{\mathfrak{L}}

\makeatother

\newtheorem{thm}{Theorem}[section]
\newtheorem{theorem}[thm]{Theorem}
\newtheorem{cor}[thm]{Corollary}

\newtheorem{prop}[thm]{Proposition}
\newtheorem{lemma}[thm]{Lemma}
\theoremstyle{definition}
\newtheorem{definition}[thm]{Definition}
\newtheorem{remark}[thm]{Remark}

\theoremstyle{definition}

\theoremstyle{remark}

\theoremstyle{remark}

\makeatletter
\def\imod#1{\allowbreak\mkern10mu({\operator@font mod}\,\,#1)}
\makeatother

\begin{document}
\title{Representations of loop extended Witt algebras}
\author{Sachin S. Sharma, Priyanshu Chakraborty, Ritesh Kumar Pandey, and S. Eswara Rao}
\date{}
\maketitle

\begin{abstract}
In this paper, we classify irreducible modules for loop extended Witt algebras with finite dimensional weight spaces. They turn out to be either modules  with uniformly bounded weight spaces or highest weight modules.
We further prove that all these modules are single point evaluation modules ($n \geq 2$). So they are actually irreducible modules for extended Witt algebras.

\end{abstract}

\tableofcontents

\bigskip

\noindent
{\bf{Notations:}} 
\begin{itemize}
\item Let $\C, \mathbb{R}, \Z$ denote the set of complex numbers, set of real numbers and set of integers respectively. 
\item Let $\Z_+$ and $\N$ denote set of non-negative integers and set of positive integers respectively. For $n \in \N$, $\C^n = \{(x_1, \ldots , x_n) : x_i \in \C, 1 \leq i \leq n\}$ and
  $\mathbb{R}^n, \Z^n, \N^{n}$ and $\Z_{+}^{n}$ are defined similarly. 
 \item For $n \geq 2$, the elements of $\C^n, \mathbb{R}^n$ and $\Z^n$ are written in boldface.
 \item For any Lie algebra $\mathfrak{g}$, and any commutative associative unital  algebra $B$ over $\C$, the notation $\mathfrak{g}(B)$ will always mean $\mathfrak{g} \otimes B$.
 \item For any Lie algebra $\mathfrak{g}$, $U(\mathfrak{g})$ will denote the universal enveloping algebra of $\mathfrak{g}$.
 \item Let $(\cdot  | \cdot)$ denote the standard inner product on $\C^n$.
\end{itemize}

\section{Introduction}
In the last four decades, the representation theory of infinite dimensional Lie algebra is well studied as it plays an important role in both Physics and Mathematics.
 Affine Lie algebras and Virasoro algebras are very basic and their representation theory has made a profound impact on the
theory of infinite dimensional Lie algebras in general. Both these infinite dimensional Lie algebras 
are generalised for several variables and are well studied. Toroidal Lie algebras are generalisations of affine Lie algebras and Witt algebras are generalisations of centerless Virasoro algebras.

Let $A_n$ be a Laurent polynomial ring in $n$ variables. Then the Lie algebra of derivations of $A_n$ is denoted by $W_n$ or $\mathrm{Der}(A_n)$ and is called a Witt algebra. Witt algebra is a classical object and studied by many mathematicians; see \cite{luzhao, ser, kgx, ybvf, xgkz}
and references theirin.
Very recently, in \cite{ybvf}, all the irreducible modules for $W_n$ with finite dimensional weight spaces have been classified by Billig and Futorny. They turn out be either uniformly bounded or ``highest weight'' modules.

Object of our paper is to study representations of a loop version of  a Witt algebra. In recent times representations of  loop algebras of  Lie algebras are extensively studied. Let $\mathfrak{g}$ be any Lie algebra and $B$ be any commutative associative unital algebra over $\C$.
Then the most general version of a loop algebra of $\mathfrak{g}$ is $\mathfrak{g} \otimes B$. The case when $B = A_n$ is very important as toroidal Lie algebra is defined as the universal central extension of $\mathfrak{g} \otimes A_n$, where $\mathfrak{g}$ is a finite dimensional 
simple Lie algebra. In \cite{R2} the theory of irreducible modules for toroidal Lie algebra with finite dimensional weight spaces is reduced to that of $\mathfrak{g} \otimes A_n$ or $\mathfrak{g}_{aff} \otimes A_n$, where $\mathfrak{g}$ and $\mathfrak{g}_{aff}$ are simple finite dimensional Lie algebra and affine
Lie algebra respectively. The case where $\mathfrak{g}$ is finite dimensional simple Lie algebra is well studied. In fact all finite dimensional irreducible modules for $\mathfrak{g} \otimes A_n$ have been classified \cite{R1}. The classification is also done where $A_n$ is replaced by any commutative associative unital algebra $B$ \cite{CFK, Mi}. The analogous problem for $\mathfrak{g}_{aff} \otimes A_n$ is solved in \cite{R1} where finite dimensional modules are replaced by integrable modules with finite dimensional weight spaces. In \cite{BR} highest weight integrable modules with finite dimension weight spaces are classified 
where $\mathfrak{g}_{aff}$ is replaced by any symmetrizable Kac-Moody algebra and $A_n$ is replaced by a finitely generated commutative associative algebra $B$ over $\C$. It is also important to mention that the theory of Weyl modules is rapidly expanding where the underlying Lie algebra is a loop algebra of a finite dimensional simple Lie algebra.

In the other direction all irreducible modules for $\mathrm{Vir} \otimes \C [t,t^{-1}]$ with finite dimensional weight spaces were classified in \cite{krg}. This result was generalised in \cite{sav}, where author replaced $\C[t,t^{-1}]$ by a finitely generated commutative associative unital algebra over $\C$.
Very recently, for $n \geq 2$, a partial classification result for the Lie algebra $(W_n \ltimes A_n) \otimes B$ is obtained in \cite{serpc}.

This paper is concerned with irreducible representations of $W_n \otimes B$ and $(W_n \ltimes A_n) \otimes B$ with finite dimensional weight spaces,  where $B$ is  a finitely generated commutative associative unital algebra over $\C$. We prove that any such module is a single point 
evaluation module (see definition \ref{evd1}). In other words these modules are irreducible modules for $W_n$ or $W_n \ltimes A_n$ and their complete classification is known \cite{ybvf, kgx}.

As pointed out earlier that the irreducible modules for $W_n$ or $W_n \ltimes A_n$  with finite dimensional weight spaces will fall in two cases \cite{kgx} : one is uniformly bounded and the other one is
highest weight modules. Weights of uniformly bounded module form a translate of a full lattice (missing at most a one point); while weights of highest weight modules are truncated from above in the one direction.

 It is crucial to note that highest weight space of highest weight modules
is a (non-trivial) uniformly bounded module for a lower rank Lie subalgebra. This fact was a driving force behind our main conclusion that  all the modules (irreducible highest weight modules with finite dimensional weight spaces) are  a single point evaluation modules.
For if it is a tensor product of  two distinct points evaluation modules then it would contain a tensor product of two uniformly bounded modules of lower rank Lie subalgebra. But as a tensor product of two nontrivial uniformly bounded irreducible modules contains infinite dimensional weight spaces (with respect lower rank Lie subalgebra), it would imply the given irreducible highest weight module contains infinite dimensional weight spaces. For $n =1$, the highest weight space is one dimensional so the above reasoning fails.

\subsection{Organisation of the paper} Our main objects of the study are the Lie algebras $W_n \otimes B$ and $(W_n \ltimes A_n) \otimes B$. But due to some technical  reasons (see the proof of Proposition \ref{propimp}), we need to consider several copies of $A_n$, i.e., $W_n \ltimes(S \otimes A_n)$ where $S$ is 
any finite dimensional abelian Lie algebra. Let us denote $\LL_{S, n} : = W_n \ltimes(S \otimes A_n)$ and $\LL_{S,n}(B) := (W_n \ltimes(S \otimes A_n)) \otimes B$ and $\LL := W_n \ltimes A_n$, where $B$ is a finitely generated commutative associative unital algebra over $\C$.

We begin the Section \ref{secb} with basic definitions. In Section \ref{secubm} we recall the known results for $W_n$ and $\LL$ along with tensor field modules (also known as Larson-Shen modules, see \cite{Lar, shen}). In this section we completely classify irreducible uniformly bounded $\LL_{S,n} (B)$-modules.
Let $V$ be an irreducible uniformly bounded $\LL_{S,n} (B)$-module.
We first note that $S$ can be reduced to at most one dimensional in Proposition \ref{prom1}. So we can consider $V$ as an irreducible uniformly bounded $\LL(B)$ or $W_n (B)$-module. In Proposition \ref{mthmub1}, we prove that the action of $A_n \otimes B$ is associative on $V$. So by application of 
\cite{serpc} Theorem 4.5, $V$ is a single point evaluation module.
 In the rest of the section we attend the case where $V$ is an irreducible uniformly bounded $W_n (B)$-module. In this case also, using the celebrated technique of $A_n$-cover \cite{ybvf}, we prove that $V$ is a single point evaluation module. 
 
 In the final section we classify all irreducible $\LL_{S, n} (B)$-modules with non-uniformly bounded but finite dimensional weight spaces. The case $n =1$ follows from \cite{wrc} and \cite{sav} with modest generalisations. For $n \geq 2$, we define a highest weight $\LL_{S,n}(B)$-modules for a given
 triangular decomposition. Then by utilising results and techniques of \cite{vmkz}, we prove in Theorem \ref{thmnu} that every irreducible $\LL_{S,n}(B)$-module with non-uniformly bounded but finite dimensional weight spaces is a highest module with an appropriate triangular decomposition. In Theorem \ref{prom2}
 we prove that the irreducible highest weight $\LL_{S,n}(B)$-modules with finite dimensional weight spaces are single point evaluation modules. This result relies on the fact that the highest weight space of highest weight module with finite dimensional weight spaces is an irreducible uniformly bounded 
 module for the lower rank Lie subalgebra which is a single point evaluation module by Theorem \ref{mthmub}. Finally, just like in the uniformly bounded case, we prove that the finite dimensional abelian Lie algebra $S$ can be reduced to  at most one dimension (Theorem \ref{mthmnub}). The following is our main result:
 \begin{thm}
 Let $V$ be a non-trivial irreducible $\LL_{S, n}(B)$-module with finite dimensional weight spaces. Then $V$ is either a uniformly bounded module or a highest weight module. Further,
 \begin{itemize}
 \item If $V$ is uniformly bounded, then $V$ can be considered as an irreducible uniformly bounded module for $\LL(B)$ or $W_n(B)$ ($n \geq 1$). Moreover, considered as an $\LL(B)$ or $W_n(B)$-module, $V$ is a single point evaluation module.
 \item If $V$ is a highest weight module, then $V$ can be considered as an irreducible highest weight module for $\LL(B)$ or $W_n(B)$ ($n \geq 2$). Moreover, considered as an $\LL(B)$ or $W_n(B)$-module, $V$ is a single point evaluation module.
 \end{itemize}
 \end{thm}

\section{Basics} \label{secb}
 Let $A_n := \C[t_1^{\pm 1}, \ldots,  t_n^{\pm 1}]$ be a Laurent polynomial ring in $n$ variables. For ${\bf{m}} = (m_1, \ldots, m_n) \in \Z^n$, we denote $t^{\bf{m}}: = t_1^{m_1} \cdots t_n^{m_n}$.  Let $W_n : = \mathrm{Der}(A_n)$ be the Lie algebra of  derivations on $A_n$, called as a Witt algebra. It is well-known that the set $\{t^{\bf{m}}d_i : {\bf{m}} \in \Z^n, 1 \leq i \leq n\}$ forms a $\C$-basis for $W_n$, where $d_i = t_i \frac{d}{dt_i}$ for $1 \leq i \leq n$. The brackets in $W_n$ are given by $[t^{\bf{m}}d_i , t^{\bf{k}}d_j] = k_i t^{\bf{m +k}}d_j - m_j t^{\bf{m+k}} d_i$.  The well-known Virasoro algebra is the unique non-trivial  one dimensional central extension of $W_1$, i.e., $\mathrm{Vir}= W_1 \oplus \C C$, with bracket operations:
 \begin{align}\label{a2.1}
[x_m,x_n]=(n-m)x_{m+n}+\delta_{m,-n}\frac{m^3-m}{12}C,\\
[x_n,C]=0,
\end{align}
for all $n,m \in \mathbb Z$ and $x_m=t_1^md_1$ for all $m \in \mathbb Z$.

 Let $H := \bigoplus_{i =1}^{n}{\C d_i}$ be an abelian subalgebra of $W_n$ which plays a role of a Cartan subalgebra for $W_n$. For ${\bb{u}} \in \C^n$ and ${\bb{r}} \in \Z^n$, denote $D(\bb{u}, \bb{r})$  $= \sum_{i =1}^{n}{u_i t^{\bb{r}} d_i}$. We then have $[D({\bb{u}}, {\bb{r}}), D({\bb{v}}, {\bb{s}})] = D(\bb{w}, \bb{r+s})$, where
 ${\bb{w}} = ({\bb{u}} | {\bb{s}}) {\bb{v}} - ({\bb{v}} | {\bb{r}}) {\bb{u}}$.
 
 Let $S$ be any finite dimensional abelian Lie algebra over $\C$. The Witt algebra $W_n$ acts on $S \otimes A_n$ by derivations: $[D({\bb{u}}, {\bb{r}}), s \otimes t^{\bb{m}}] = ({\bb{u}} |{ \bb{m}}) s \otimes t^{\bb{m+r}}$, for all $s \in S$. The emerging Lie algebra $\mathfrak{L}_{S, n}: = W_n \ltimes (S \otimes A_n)$ is 
 called an extended Witt algebra. In a special case when $S$ is one dimensional, we denote the corresponding Lie algebra by $\LL_n: = W_n \ltimes A_n$ and when there is no ambiguity about $n$ we will simply denote it by $\LL$. Let $B$ be any finitely generated commutative associative algebra with unity over $\C$.  Then the Lie algebra $\mathfrak{L}_{S, n}(B): = \big(W_n \ltimes (S \otimes A_n)\big) \otimes B$ is called a loop extended Witt algebra with the brackets:
 $[u_1 \otimes b_1, u_2 \otimes b_2] = [u_1, u_2] \otimes b_1 b_2$ for $u_i \in \mathfrak{L}_{S, n}$ and $b_i \in B$, $i = 1,2$. The abelian subalgebra $H_1 :=  H \oplus \C$ plays a role of a Cartan subalgebra for $\mathfrak{L}_{S, n}(B)$.
 We also need to introduce the twisted Heisenberg Virasoro algebra
HVir which is a Lie algebra with a basis $\{x_i, I(j), C_D, C_{DI}, C_{I} | i, j \in \Z\}$ and the following bracket operations:
$$[x_i, x_j] = (j-i)x_{i+j} + \delta_{i, -j} \frac{i^3 - i}{12}C_D ,$$
$$[x_i, I(j)] = j I(i+j) + \delta_{i, -j}(I^2 + i) C_{DI}, $$
$$[I(i), I(j)] = i \delta_{i, -j} C_{I},$$
$$[\mathrm{HVir}, C_{D}] = [\mathrm{HVir}, C_{DI}] = [\mathrm{HVir}, C_{I}] = 0.$$
 
 Now, we recall some basic definitions. Let $\mathfrak{g}$ be any Lie algebra and let $\mathfrak{h}$ be its Cartan subalgebra.
 \begin{definition} A $\mathfrak{g}$-module $V$ is called a trivial $\mathfrak{g}$-module, if $x.v = 0$ for all $x \in \mathfrak{g}$ and for all $v \in V$. Hence if a $\mathfrak{g}$-module $V$ is not a trivial module, then there exists a $x \in \mathfrak{g}$ and $0 \neq v \in V$ such that
 $x.v \neq 0$.
 \end{definition}
 \begin{definition} A $\mathfrak{g}$-module $V$ is called a weight module if $V = \bigoplus_{\lambda \in \mathfrak{h}^{*}}{V_{\lambda}}$, where $V_{\lambda} = \{v \in V: h.v = \lambda(h).v \,\, \forall \,\, h \in \mathfrak{h}\}$.
 \end{definition}
 The set $\{\lambda \in \mathfrak{h}^*: V_{\lambda} \neq 0\}$ is called support of a module $V$ and is denoted as $\mathrm{Supp}(V)$. The elements of $\mathrm{Supp}(V)$ are called weights of $V$ and the subspaces $V_{\lambda}$ for $\lambda \in \mathrm{Supp}(V)$ are called weight spaces of $V$.
 \begin{definition}
 A $\mathfrak{g}$-module $V$ is called a Harish-Chandra module if it is a weight module and all its weight spaces are finite dimensional.
 \end{definition}
 \begin{definition}
 A $\mathfrak{g}$-module $V$ is called uniformly bounded if it is a weight module and dimensions of its all weight spaces are uniformly bounded, i.e., $\exists N \in \N$ such that $\mathrm{dim}(V_{\lambda}) < N$ for all $\lambda \in \mathrm{Supp}(V)$.
 \end{definition}
 \begin{definition} \label{evd1}
 Let $B$ be a commutative associative unital finitely generated algebra over $\C$. Let $V$ be any $\mathfrak{g} \otimes B$-module. Then $V$ is called a single point evaluation $\mathfrak{g} \otimes B$-module if
 there exists an algebra homomorphism $\eta:B \mapsto \C$ such that $x \otimes b.(v) = \eta(b)(x\otimes1).v$ for all $x \in \mathfrak{g}, b \in B$ and $v \in V$. In other words if $(V, \rho)$ is a $\mathfrak{g} \otimes B$-representation, then $V$ is a single 
 point evaluation module if the Lie algebra homomorphism $\rho: \mathfrak{g} \otimes B \rightarrow \mathrm{End}(V)$ factors through $\mathfrak{g} \otimes B/ \mathfrak{g} \otimes \mathfrak{m} \cong \mathfrak{g} \otimes B/ \mathfrak{m}$, where $\mathfrak{m}$ is a maximal ideal of $B$.
 \end{definition}
 It is easy to see that an irreducible single point evaluation $\mathfrak{g} \otimes B$-module is an irreducible $\mathfrak{g} \otimes \C \cong \mathfrak{g}$-module.
 \begin{definition}
 Let $(V, \rho)$ be a representation of $\mathfrak{g} \otimes B$. Then $V$ is called a single point generalised evaluation module if the map $\rho: \mathfrak{g} \otimes B \rightarrow \mathrm{End}(V)$ factors through $\mathfrak{g} \otimes B/ \mathfrak{g} \otimes \mathfrak{m}^k \cong \mathfrak{g} \otimes B/ \mathfrak{m}^k$ for some $k \in \N.$, where $\mathfrak{m}$ is a maximal ideal of $B$.
 \end{definition}

 \section{Uniformly bounded $\mathfrak{L}_{S, n}(B)$-modules} \label{secubm}
 In this section we classify all irreducible uniformly bounded modules for $\mathfrak{L}_{S, n}(B)$. We begin with:
  \subsection{A brief survey of some Known results} 
  It is well known that there is a class of intermediate modules $V_{\alpha, \beta}$ for $\mathrm{Vir}$ with two parameters $\alpha, \beta \in \mathbb C$. As a vector space $V_{\alpha, \beta}=\displaystyle{ \bigoplus_{n \in \mathbb Z}}\mathbb C v_n$ and $\mathrm{Vir}$ action on $V_{\alpha, \beta}$ is given by,
\begin{align}\label{a2.3}
x_n.v_k=(\alpha+k+n\beta)v_{k+n},\\
C.v_k=0,
\end{align}
for all $n,k \in \mathbb Z$. The following result is well-known:
\begin{lemma} \cite{kr} \label{prom 4}
\item[1.] The $\mathrm{Vir}$-module $V_{\alpha, \beta} \simeq V_{\alpha+m, \beta}$ for all $m \in \mathbb Z$.
\item[2.] The $\mathrm{Vir}$-module $V_{\alpha, \beta}$ is irreducible if and only if $\alpha \notin \mathbb Z $  and $\beta \notin \{0,1\}$.
\item[3.] $V_{0,0}$ has a unique trivial proper submodule $\mathbb Cv_0$ and denote $V_{0,0}'=V_{0,0}/\mathbb Cv_0$.
\item[4.] $V_{0,1}$ has a unique non-zero proper submodule $V'_{0,1}=\displaystyle{\bigoplus_{i\neq 0}} \mathbb C v_i$. 
\item[5.]  $V'_{\alpha,0}\simeq V'_{\alpha,1}$ for all $\alpha \in \mathbb C$.
\end{lemma}
For convenience we denote $V'_{\alpha, \beta}:= V_{\alpha, \beta}$ if $V_{\alpha, \beta}$ is irreducible. We have the following:
\begin{thm} \cite{om} \label{prom 5}
Any nontrivial irreducible uniformly bounded $\mathrm{Vir}$-module is  isomorphic to $V'_{\alpha, \beta}$ for some $\alpha, \beta \in \C$.
\end{thm}
Now, consider the Lie algebra $\mathrm{Vir} \otimes B$, where $B$ is finitely generated commutative associative unital algebra over $\C$. For $B = \C[t, t^{-1}]$ the following result was proved in \cite{krg}:

\begin{thm}
Let $V$ be an irreducible uniformly bounded $\mathrm{Vir} \otimes \C[t, t^{-1}]$-module. Then $V$ is a single point evaluation module.
\end{thm}
Savage generalised the above result for any finitely generated commutative associative unital algebra over $\C$ \cite{sav}:
\begin{thm}
Let $V$ be an irreducible uniformly bounded $\mathrm{Vir} \otimes B$-module. Then $V$ is a single point evaluation module.
\end{thm}

 We now move our attention to $n \geq 2$ case.
 Let us recall the tensor field  module on a torus $T(U, {\bm \gamma}): = t^{{\bm \gamma}} A_n \otimes U$, where $U$ is a finite dimensional irreducible $\mathfrak{gl}_n$ representation and 
 ${\bm \gamma} \in \C^n$, which has the following $W_n$-module structure:
 $$t^{\bb m}d_i. (t^{{\bm \gamma}+ {\bb{s}}}) = (\gamma_i + s_i) t^{{\bm \gamma} + {\bb s}+ {\bb m}} \otimes u + \sum_{j = 1}^{n}{m_j t^{{\bm \gamma} + {\bb s}+ {\bb m}}} \otimes E_{j i} u,$$
 where ${\bb m, s} \in \Z^n,$ $u \in U,$  $1 \leq i \leq n$ and $E_{ji}$ is the $n \times n$ matrix whose $(j,i)$th entry is 1 and all others are zero. Let $V$ be standard representation of $\mathfrak{gl}_n$ of dimension $n$. We know that  for $ k = 0, \ldots, n$, $\Lambda^{k} \,V$ is an irreducible finite dimensional representation of $\mathfrak{gl}_n$. Then the tensor field modules  $ T(\Lambda^{k} \,V, {\bm \gamma}) =: \Omega^{k}({\bm \gamma})$ are also called modules of differential $k$-forms and they form  the de Rham complex:
 $$\Omega^{0} ({\bm \gamma}) \xrightarrow{d} \Omega^{1} ({\bm \gamma}) \xrightarrow{d} \cdots \xrightarrow{d}   \Omega^{n}({\bm \gamma}) .$$
 The differential map $d: \Omega^{k}({\bm \gamma}) \rightarrow \Omega^{k+1}({\bm \gamma})$ is $W_n$-module map for $k = 0, \ldots, n-1$, so we have that $\Omega^{k}({\bm \gamma})$  is reducible $W_n$-module for $k = 1, \dots, n-1$ and $\Omega^{0}({\bm \gamma}), \Omega^{n}({\bm \gamma})$ 
 are reducible iff ${\bm \gamma} \in \Z^n$. We have the following: 
 \begin{thm} \label{prom 2} \cite{ybvf} 
 Let $V$ be a non-trivial irreducible uniformly bounded $W_n$-module. Then one of the following statements hold:
 \begin{enumerate}[label= \((\arabic*)\)]
\item Either $V \cong T(U, {\bm \beta})$, where $U$ is an irreducible finite dimensional $\mathfrak{gl}_n$-module which is not an exterior power of a standard $n$-dimensional $\mathfrak{gl}_n$-module and ${\bm \gamma} \in \C^n$, or
\item $V \cong d(\Omega^k)({\bm \beta}) \subseteq \Omega^{k+1}({\bm \beta})$, where $0 \leq k <n$ and ${\bm \beta} \in \C^n$.
\end{enumerate}
\end{thm}
Now, consider an $\LL$-module $V$. We say the action of $A_n$ is associative on $V$ if $t^{{\bb m}} t^{{\bb s}}.v = t^{{\bb m+s}}.v $ for all ${\bb m, s} \in \Z^n$ and $v \in V$. The following result was proved in \cite{er}:
\begin{thm}
Let $V$ be an irreducible $\LL$-module with finite dimensional weigh spaces with an associative action of $A_n$ and $t^{{\bb 0}}.v = v$ for all $v \in V$. Then $V \cong T(U, {\bm \gamma})$, where $U$ is an irreducible
finite dimensional representation of $\mathfrak{gl}_n$  and ${\bm \gamma} \in \C^n$, with the following actions:
$$t^{\bb m}d_i. (t^{{\bm \gamma}+ {\bb{s}}}) = (\gamma_i + s_i) t^{{\bm \gamma} + {\bb s}+ {\bb m}} \otimes u + \sum_{j = 1}^{n}{m_j t^{{\bm \gamma} + {\bb s}+ {\bb m}}} \otimes E_{j i} u,$$
$$t^{{\bb m}}. t^{{\bm \gamma} + {\bb s}} \otimes u = t^{{\bm \gamma} + {\bb m + s}} \otimes u.$$
\end{thm}
In the following result \cite{kgx} it is proved that if $V$ is an irreducible uniformly bounded $\LL$-module and if $t^{\bb 0}$ acts by a non-zero scalar then the action of $A_n$ is associative on $V$. In the case where $t^{\bb 0}$ acts trivially on $V$,
$A_n$ acts trivially on $V$, and so $V$ is an irreducible uniformly bounded $W_n$-module. More precisely:
\begin{thm} \cite[Theorem 3.3]{kgx} \label{prom 3}
Let $V$ be an irreducible uniformly bounded $\LL$-module. Then 
\begin{enumerate}[label= \((\arabic*)\)]
\item If $t^{\bb 0}$ acts as a non-zero scalar then $A_n$ acts associatively on $V$ and  $ V \cong T(U, {\bm \gamma})$, where $U$ is an irreducible finite dimensional $\mathfrak{gl}_n$-module and ${\bm \gamma} \in \C^n$.
\item If $t^{\bb 0}$ acts as a zero then $A_n. V = 0$ and hence $V$ is an irreducible $W_n$-module.
\end{enumerate}

\end{thm}
The following result deals with the classification of irreducible Harish-Chandra $\LL(B)$-modules under some conditions:
\begin{theorem} \cite[Theorem 4.5]{serpc}
Let $V$ be an irreducible weight module for $\LL(B)$ with finite dimensional weight spaces. Assume that the action of $A_n\otimes B$ is associative on $V$ and $t^{\bb 0} .v=v$ for all $v \in V.$ Then $V$ is a single point evaluation $\LL(B)$ module and 
hence an irreducible $\LL$-module.
 \end{theorem}
 Finally, we recall some results on $\mathrm{HVir}$ which will play an important role in this paper.
 For the Lie algebra HVir, the subalgebra spanned by $\{ x_0, I(0), C_{D}, C_{DI}, C_{I}\}$ plays a role of Cartan subalgebra. For $a, b, c \in \C$ consider HVir-module $W(a,b, c)$ with a basis $\{w_i | i \in \Z \}$ and the following action:
$$x_j. w_i = (a+i+bj)w_{i+j}, I(j).w_{i}= c w_{i+j}, C_D .w_i = C_{DI}.w_i = C_I . w_i = 0.$$  It is known that $W(a,b,c)$ is reducible iff $a \in \Z$, $b$ is either 0 or 1 and $c =0$. Let $W'(a,b,c)$ be a unique non-trivial irreducible sub-quotient of  
$W(a,b,c)$. The following is an important result from \cite[Theorem 4.4]{lrkz}.
\begin{thm} \label{thm1} 
Let $V$ be an irreducible  HVir-module with uniformly bounded weight spaces. Then $V \cong W'(a,b,c)$ for some $a,b,c \in \C$.
\end{thm}
Now consider $\overline{\mathrm{HVir}}$ be the  Lie algebra $\frac{\mathrm{HVir}}{\C C_D \oplus \C C_{DI} \oplus \C C_I}$ with the Cartan subalgebra $\C x_0 \oplus \C I(0)$. It is easy to see that as a Lie algebra
$\overline{\mathrm{HVir}}$ is isomorphic to $\big(\bigoplus_{j \in \mathbb Z} D({{\bm \beta}},  j{\bb s}) \ltimes \bigoplus_{k \in \Z}t^{k{\bb s}} \big)$, where $0 \neq {\bb s} \in \Z^n$ and ${\bm \beta} \in  \C^n$ such that $({\bm \beta}\, | \,{\bb s}) \neq 0$ with  
$x_i \mapsto \frac{D({\bm \beta}, i{\bb s})}{({\bm \beta}\, | \,{\bb s})}$ and $I(i) \mapsto \frac{t^{i{\bb s}}}{({\bm \beta}\, | \,{\bb s})}.$
Taking into account of the above isomorphism we denote the Lie algebra  $\big(\bigoplus_{j \in \mathbb Z} D({{\bm \beta}},  j{\bb s}) \ltimes \bigoplus_{k \in \Z}t^{k{\bb s}} \big)$ by $\overline{\mathrm{HVir}}({\bm \beta},  {\bb s}).$
 This paper is concerned with the classification of all the irreducible Harish-Chandra modules of $\mathfrak{L}_{S, n}(B)$. 
The following lemma will be used in this section. The proof of this lemma follows along similar lines as that of proof of Lemma 3.2 in cite{kgx}.
\begin{lemma}\label{lemm1}
Let $V$ be an irreducible $\LL(B)$-module such that $u.v = 0$ for some $u \in U(A_n \otimes B)$ and some non-zero vector $v$ in V. Then $u$ is locally nilpotent on $V$.
\end{lemma}
\subsection{Reduction to $\LL(B)$ or $W_n (B)$-modules} 
 We start with the following:
 \begin{prop} \label{prom1}
 Let $V$ be an irreducible uniformly bounded $\mathfrak{L}_{S, n}(B)$-module. 
 \begin{enumerate}[label= \((\arabic*)\)]
 \item Let $(S \otimes \C t^{\bb 0}) \otimes 1$ act non-tivially on V. Then there exists a subspace $\overset{\circ}{S}$ of  $S$ of co-dimension $1$ such that $(\overset{\circ}{S} \otimes A_n) \otimes B$ acts trivially on V.
  In particular $V$ is an irreducible $\LL(B)$-module.
 \item  If $(S \otimes \C t^{\bb 0}) \otimes 1$ acts trivially, then $V$ is an irreducible $W_n (B)$-module.
 \end{enumerate}
  \end{prop}
  \begin{proof}
  Note that $(S \otimes \C t^{\bb 0} )\otimes 1$ is central in $\mathfrak L_{S,n}(B)$, hence its elements act as  scalars on $V.$ Since $(S \otimes \C t^{\bb 0}) \otimes 1$ acts non-trivially on $V$,  there exists a co-dimension one subspace $\overset{\circ}{S}$ of $S$ such that $(\overset{\circ}{S} \otimes \C t^{\bb 0}) \otimes 1.V=0$.
Let us choose a ${\bm \gamma} \in \mathbb C^n$ such that $({\bm \gamma}, {\bb s})\neq 0$ for all non-zero ${\bb s} \in \mathbb \Z^n$. Fix a non-zero ${\bb m} \in \mathbb Z^n$ and some $y \in \overset{\circ}{S}.$ Now consider the Heisenberg-Virasoro algebra  $\mathrm{HVir}({\bm \gamma}, {\bb m},y)= \mathrm{span} \{D({\bm \gamma} \,| i {\bb m}), y\otimes t^{i {\bb m}} \otimes 1 \, \, | \,\, i \in \mathbb Z  \}$. For ${\bb r} \in \Z^n$ consider the uniformly bounded $\mathrm{HVir}({\bm \gamma}, {\bb m}, y)$-module $\displaystyle{\bigoplus_{i \in \mathbb Z}V_{\lambda+ {\bb r} +
i {\bb m}}}$. Now, following the proof of Theorem 3.3(2) of \cite{kgx}, which relies on Theorem \ref{thm1}, we find a least positive integer $N$ such that $(y\otimes t^{\bb m} \otimes 1)^N.V=0.$ But then we have
$$D({\bm \gamma},{\bb k-m})\otimes b.(y\otimes t^{\bb m} \otimes 1)^N.V= ({\bm \gamma} | {\bb m})Ny\otimes t^{\bb k}\otimes b.(y\otimes t^{\bb m}\otimes 1)^{N-1}.V=0,  $$
for all ${\bb k} \in \mathbb Z^n, b \in B.$ In particular, there exists a non-zero $u \in V$  such that $y \otimes t^{\bb r} \otimes b.u=0$ for all ${\bb r} \in \mathbb Z^n , b \in B.$  Then consider the set $W: =\{v \in V: y \otimes t^{\bb r} \otimes b.v=0$ for all ${\bb r} \in \mathbb Z^n, b \in B\}$, which is a non-zero submodule of $V$ so must be $V$ itself.
This completes the proof of (1). Similar argument as that of proof of (1) will work for (2).
  \end{proof}
  \subsection{$t^{{\bb 0}} \otimes 1$ acting as an identity}
  By above proposition we can assume $V$ to be an irreducible uniformly bounded $\LL(B)$-module. In the next proposition we work with the assumption that $t^{\bb 0} \otimes 1$ acting non-zero scalar say c. By taking a Lie algebra automorphism of $A_n$ which assigns $t^{\bb s} \mapsto \frac{t^{\bb s}}{c}$, we may assume $c = 1$. We have the following:
  
  \begin{theorem}\label{mthmub1}
Let $V$ be an irreducible uniformly bounded module for $\mathfrak{L}(B)$. If $t^{\bb 0}\otimes 1$ acts as an identity on $V$, then $V$ is a single point evaluation module.
\end{theorem}
\begin{proof}
We will use Theorem 4.5 of \cite{serpc}, by which we will be done if we show that $A_n \otimes B$ acts associatively on $V$, i.e., $t^{\bb r}\otimes b_1 t^{{\bb s}} \otimes b_2.v=t^{{\bb r+s}}\otimes b_1b_2.v$ for all ${\bb r,s} \in \mathbb Z^n, b_1,b_2 \in B, v \in V.$
First as $1\otimes B$ is central in $\mathfrak{L}(B)$ there exists a linear map $\psi:B \to \mathbb C$ such that $\psi(1)=1$. If $b \in \mathrm{Ker} \,\psi$, then $1 \otimes b. v = 0$ for all $v \in V$.

\noindent
{\bf Claim :} $(t^{\bb r} \otimes bb').V=0$, for all $b \in \mathrm{Ker}\, \psi , b' \in B,{\bb r} \in \mathbb Z^n$. In particular, $\mathrm{Ker} \,\psi$ is a maximal ideal of $B$.

\noindent
Let ${\bm \gamma} \in \mathbb C^n$ be such that $({\bm \gamma} \,|\, {\bb s})\neq 0$ for all non-zero ${\bb s} \in \mathbb Z^n$. Fix $0 \neq {\bb m} \in \Z^n$ and $b \in \mathrm{Ker}\, \psi$ and consider the Lie algebra $\mathrm{HVir}({\bm \gamma},{\bb m},b) = \mathrm{span} \, \{D({\bm \gamma}, i{\bb m}), t^{i {{\bb m}}}\otimes b \, \, | \,\, i \in \mathbb Z  \}$.
Now using similar arguments as that of Theorem 3.2 (2) of \cite{kgx}, find a smallest positive integer $N_0$ such that $(t^{\bb m} \otimes b)^{N_0}.V=0$ for all ${ \bb m} \in \mathbb Z^n.$ Then we have
 $$D({\bm \gamma}, {\bb k -m}) \otimes b'.(t^{\bb m} \otimes b)^{N_0}.V = ({\bm \gamma} \,|\, {\bb s})N_0 \, (t^{\bb k} \otimes bb') \, (t^{\bb m} \otimes b)^{N_0 -1}.V = 0 ,$$
 for all ${\bb k} \in \mathbb Z^n$ and $b' \in B.$
 Hence we have  $V' = \{v \in V \mid  t^{\bb r} \otimes bb'.v= 0 \, \forall \, \, {\bb r} \in \Z^n, \, \forall \, b' \in B \}$ is a non-zero submodule of $V$. This proves the claim. 
 Note that $b -\psi(b)1 \in \mathrm{Ker}\, \psi$, therefore from the above claim it follows that, $t^{\bb r} \otimes b.v=\psi(b)t^{\bb r}.v$ for all ${\bb r} \in \mathbb Z^n, b \in B, v \in V$.
  Since $\mathrm{Ker} \,\,\psi$ is maximal ideal in $B$ of co-dimension one there exists a non-zero algebra homomorphism $\eta:B \to \mathbb C$ such that $\eta=c \psi$. Now $\eta(1)=\eta(1)\eta(1)$ with the facts $\psi(1)=1$ and $\eta$ non-zero gives us $c=1.$ Hence we have,
 \begin{align}
     t^{\bb r} \otimes b.v=\eta(b)t^{\bb r}.v,
 \end{align}
  for all ${\bb r} \in \mathbb Z^n, b \in B, v \in V$ and for some algebra homomorphism $\eta$. Hence to prove the associativity of $A_n \otimes B$ we have to prove the associativity of $A_n \otimes 1.$ Now using the techniques of
 proof of Theorem 3.2 (1) of \cite{kgx}, which involve Lemma \ref{lemm1}, we can find a non-zero $w \in V$ such that $t^{\bb m}.t^{\bb n}.w=t^{{\bb m+n}}.w$ for all ${\bb m,n} \in \mathbb Z^n$. Finally note that $W=\{v \in V: t^{\bb m}.t^{{\bb n}}.v=t^{{\bb m+n}}.v\} $ is a non-zero $\mathfrak{L}(B)$-submodule of $V$. Hence the proof is complete.
\end{proof}
 \subsection{$t^{{\bb 0}} \otimes 1$ acting trivially}
Now we concentrate on the case where $t^{{\bb 0}}\otimes 1$ acts trivially on $V$. By Proposition \ref{prom1}(2) $V$ is an irreducible $W_n (B)$-module. We have the following:
\begin{theorem}\label{t0.3}
Let $V$ be a uniformly bounded irreducible module for $W_n (B)$
with finite dimensional weight spaces. Then there exists a co-finite ideal $J$ of $B$ such that $W_n \otimes J.V=0.$

\end{theorem}
\begin{proof}
As $V$ is irreducible, we have 
 $V=\displaystyle{\bigoplus_{{\bb m} \in \mathbb{Z}^n}}V_{{\bm \alpha}+ {\bb m}},$ where ${\bm \alpha} \in \mathbb{C}^n$ with dim $V_{\bm \alpha} \neq 0$ and $V_{{\bm \alpha}+ {\bb m}} = \{ v \in V: D({\bb u}, {\bb 0})v=({\bb u}, {\bm \alpha}+ {\bb m})v $ for all ${\bb u} \in \mathbb{C}^n \}$.
Let $N \in \mathbb{N}$ such that $\mathrm{dim}\,(V_{\alpha+m}) \leq N$ for all $ {\bb m} \in \mathbb{Z}^n$.
 For ${\bb 0} \neq {\bb r} \in \mathbb{Z}^n$, let $I_{(1,{\bb r})} := \{ b \in B : D({\bb e_1, r}) \otimes b.V_\alpha = 0\}$. Then $I_{(1,{\bb r})}$ is the kernel of the linear map
    $$ B \to \mathrm{Hom} (V_{\bm \alpha}, V_{{\bm \alpha}+ {\bb r}}), 
 b \to (v \to D({\bb e_1, r})bv),$$
and hence we have $\mathrm{dim}(B/I_{(1,{\bb r})}) \leq N^2$. Note that  $[D({\bb e_j, 0}) \otimes s, D({\bb e_1, r}) \otimes b] = r_jD({\bb e_1, r}) \otimes bs$ implies that $I_{(1,{\bb r})}$ is an ideal of B for all $r \neq 0$.

\noindent
{\bf Claim 1 :} $I_{(1,{\bb e_1})}I^j_{(1,{\bb e_2})} \subseteq I_{(1, {\bb e_1} + j {\bb e_2})}$, for all $j \geq 1.$

\noindent
Note that $D({\bb e_1, e_1 + e_2}) \otimes bc = -[D({\bb e_1, e_1}) \otimes b, D({\bb e_1, e_2}) \otimes c].$ This implies that the claim is
true for $j = 1.$ Now it will follow by induction that the claim 1 is true for all $j \geq 1.$

\noindent
{\bf Claim 2 :} $ I_{(1,{\bb e_1})}{I^{N^2}_{(1,{\bb e_2}) }}
\subseteq I_{(1, {\bb e_1} + j {\bb e_2})}$ for all $j \geq 1$. 

\noindent
Fix some $j \in \mathbb N$ and consider the chain of subspaces
$$\frac{B}{I_{(1, {\bb e_1} + j {\bb e_2})}} \supseteq \frac{I_{(1,{\bb e_1})}+I_{(1, {\bb e_1} + j {\bb e_2})}}{I_{(1, {\bb e_1} + j {\bb e_2})}}  \supseteq \frac{I_{(1,{\bb e_1})}I_{(1,{\bb e_2})}+I_{(1, {\bb e_1} + j {\bb e_2})}}{I_{(1, {\bb e_1} + j {\bb e_2})}} \supseteq \frac{I_{(1,{\bb e_1})}I^2_{(1,{\bb e_2})}+I_{(1, {\bb e_1} + j {\bb e_2})}}{I_{(1, {\bb e_1} + j {\bb e_2})}} \supseteq  \cdots$$
Since $\mathrm{dim}( B/I_{(1, {\bb e_1} + j {\bb e_2})}) \leq N^2,$ there exists an $l \leq N^2$ such that $I_{(1,{\bb e_1})}I^l_{(1,{\bb e_2})} + I_{(1, {\bb e_1} + j {\bb e_2})} =
I_{(1,{\bb e_1})}I^{l+1}_{(1,{\bb e_2})} + I_{(1, {\bb e_1} + j {\bb e_2})}.$ This implies that, $I_{(1,{\bb e_1})}I^m_{(1,{\bb e_2})} + I_{(1, {\bb e_1} + j {\bb e_2})} =
I_{(1,{\bb e_1})}I^l_{(1,{\bb e_2})} + I_{(1, {\bb e_1} + j {\bb e_2})}$ for all $m \geq l$ and $l \leq N^2.$
Now $I_{(1,{\bb e_1})}I^j_{(1,{\bb e_2})} \subseteq I_{(1, {\bb e_1} + j {\bb e_2})} $ implies that $I_{(1,{\bb e_1})}\, I^l_{(1,{\bb e_2})} + I_{(1, {\bb e_1} + j {\bb e_2})} = I_{(1, {\bb e_1} + j {\bb e_2})}$ for $1 \leq j \leq l$. Hence we have $I_{(1,{\bb e_1})}I^ l_{(1,{\bb e_2})}+ I_{(1, {\bb e_1} + j {\bb e_2})} = I_{(1, {\bb e_1} + j {\bb e_2})}, \, \forall \, j \geq 1,$ this proves the Claim 2.

\noindent
{\bf Claim 3 :} $I_{(1,{\bb e_1})}I^{N^2}_{(1,{\bb e_2})}I^{N^2}_{(1,-{\bb e_2})} \subseteq I_{(1,{\bb e_1}+r {\bb e_2})}$ for all $r \in \mathbb Z .$
Note that $I_{(1,{\bb e_1})}I^{N^2}_{(1,{\bb e_2})}I_{(1,-{\bb e_2})} \subseteq I_{(1,{\bb e_1+je_2-e_2})} $ for all $j \geq 1,$ because
$$[D({\bb e_1 , e_1} + j\,{\bb e_2})\otimes ab, D({\bb e_1}, -{\bb e_2})\otimes c] = -D({\bb e_1, e_1} + j{ \bb e_2 - e_2})\otimes abc $$
for all $a \in I_{(1,{\bb e_1})}, b \in I^{N^2}_{(1,{\bb e_2})}, c \in I_{(1,-{\bb e_2})}.$ Hence by induction we have $I_{(1,{\bb e_1})}I^{N^2}_{(1,{\bb e_2})}I^k_{(1,-{\bb e_2}) }\subseteq I_{(1,{\bb e_1}+j\,{\bb e_2} -k \,{\bb e_2})}$ for all $j \geq 1, k \geq 1.$ Now with the same procedure as like claim 2 we
have $I_{(1,{\bb e_1})}I^{N^2}
_{(1,{\bb e_2})}I^{N^2}_{(1,-{\bb e_2})} \subseteq I_{(1, {\bb e_1}+j{\bb e_2}-k{\bb e_2})}$ for all $j,k \geq  1.$ Hence claim 3 follows.
Now it is easy to see with similar arguments that $I_{(1,{\bb e_1})}\displaystyle{\prod_{j=2}^{n}I^{N^2}_{(1,{\bb e_j})}I^{N^2}_{(1,-{\bb e_j})}} \subseteq I_{(1, {\bb e_1+r})}$ for all ${\bb r} \in \mathbb Z^{n-1}.$

\noindent
{\bf Claim 4:} $I_{(1,{\bb e_1})}I^{N^2}_{(1,2{\bb e_1})} \displaystyle{\prod_{j=2}^{n}I^{N^2}_{(1,{\bb e_j})}I^{N^2}_{(1,-{\bb e_j})}}\subseteq I_{(1,{\bb e_1}+2k{\bb e_1}+r {\bb e_2})}$ for all ${\bb r} \in \mathbb Z^{n-1}$, $k \geq 1.$
Note that $I_{(1,{\bb e_1})}I_{(1,2{\bb e_1})}\displaystyle{\prod_{j=2}^{n}I^{N^2}_{(1,{\bb e_j})}I^{N^2}_{(1,-{\bb e_j})}} \subseteq I_{(1, {\bb e_1}+2{\bb e_1}+r {\bb e_2})}$ for all ${\bb r} \in \mathbb Z^{n-1}$, because
$$ [D({\bb e_1, e_1 + r}) \otimes ab, D({\bb e_1}, 2{\bb e_1})\otimes c] = D({\bb e_1, e_1} + 2{\bb e_1 + r}) \otimes abc.$$
Hence by induction we have $I_{(1,{\bb e_1})}I^k_{(1,2{\bb e_1})}\displaystyle{\prod_{j=2}^{n}I^{N^2}_{(1,{\bb e_j})}I^{N^2}_{(1,-{\bb e_j})}} \subseteq I_{(1,{\bb e_1}+2k{\bb e_1+r})}$ for all $r \in \mathbb Z, k \geq 1.$ Now applying the method like claim 2 we have claim 4.

\noindent
{\bf Claim 5:} 
Set $ I_1 = I_{(1,{\bb e_1})}I^{N^2}
_{(1,2{\bb e_1})}I^{N^2}_{(1,-3e_1)}\displaystyle{\prod_{j=2}^{n}I^{N^2}_{(1,{\bb e_j})}I^{N^2}_{(1,-{\bb e_j})}} \subseteq I_{(1,{\bb e_1}+2k{\bb e_1}-3l{ \bb e_1+r})}$
for all ${\bb r} \in
\mathbb Z^{n-1}, k,l \geq 1.$ This claim is clear with the previous method.
Note that the set $\{1 + 2k - 3l : k, l \in \mathbb N \} = \mathbb Z.$ Hence $I_1 \subseteq I_{(1,{\bb r})},$ for all ${\bb r} \in 
\mathbb Z^n \setminus  \{0\}$. 

\noindent
{\bf Claim 6 :} $D({\bb e_1, r}) \otimes  J_1.V_\alpha = 0,$ for all ${\bb r} \in \mathbb Z^n$, $J_1=I_1^2$.
Note that $J_1 \subseteq I_1,$ so $D({\bb e_1, r}) \otimes  J_1.V_\alpha = 0,$ for all ${\bb r} \in \mathbb Z^n \setminus \{0\}.$ For ${\bb r = 0},$ just note
that $$[D({\bb e_1, e_1}) \otimes a, D({\bb e_1, -e_1})\otimes b] = 2D({\bb e_1, 0}) \otimes ab.$$ This proves the claim.
In a similar way we get co-finite ideals $J_i$ such that $D({\bb e_i, r}) \otimes  J_i.V_\alpha = 0$ for all ${\bb r} \in \mathbb Z^n, 1 \leq i \leq n.$ Then for $J := J_1J_2 \cdots J_n$ we have $W_n \otimes J.V_\alpha= 0$
and $W = \{v \in V : W_n \otimes J.v = 0\}$  is a non-zero submodule of $V$, so must be equal to $V$.
\end{proof}
To proceed further, we need to recall  the definitions of $A_nW_n$-module structure and an $A_n$-cover for any $W_n$-module $M$ (see \cite{ybvf} for more details).
\begin{definition}
A $W_n$-module $M$ is said to have an $A_nW_n$-module structure if it is also a module for an associative algebra $A_n$ and satisfies the following:
$$X(fv) = (Xf)v + f(Xv),  X \in W_n, f \in A_n, v \in M.$$
\end{definition}
In other words, $M$ is called an $A_n W_n$-module if $M$ is an $\LL$-module with an associative action of $A_n$, i.e., $t^{{\bb 0}}.v = v$ and $t^{{\bb m+n}}.v = t^{{\bb m}}t^{{\bb n}}.v$ for all $v \in M$.
It is easy to see that $W_n$ has an $A_n W_n$-module structure with an $A_n$ action given as follows:
$$t^{{\bb m}}. t^{{\bb s}} d_i = t^{{\bb m+s}} d_i.$$
\begin{definition}
An $A_n$-cover of a $W_n$-module $M$ is the subspace $\widehat{M} \subset \mathrm{Hom}(A_n, M)$ spanned by the set $\{\psi(X,v):X\in W_n, v \in M\}$, where $\psi(X,v) \in \mathrm{Hom}(A_n,M)$ is given by 
$\psi(X,v)(f)=(fX)v  $ for all $f \in A_n$.
\end{definition}
An $A_n$-cover of a $W_n$-module $M$ acquires an $A_n W_n$-module structure:
\begin{align*}
X.\psi(Y, v) &= \psi([X,Y], v) + \psi(Y, X.v)\\
g.\psi(Y, v) &= \psi(g Y, v), \,\, X,Y \in W_n, g \in A_n. v \in M,
\end{align*}
and it is useful in the following sense \cite{ybvf} Theorem 4.8:
\begin{theorem}
Let $V$ be an irreducible uniformly bounded module for $W_n$. Then its $A_n$ cover  $\widehat V$ is a uniformly bounded $\LL$-module.
\end{theorem}

\begin{theorem} \label{mthmub2}
Any uniformly bounded irreducible $W_n\otimes B$-module is a single point evaluation module.
\end{theorem}

\begin{proof}
 Let $V$ be an irreducible uniformly bounded $W_n (B)$-module. Then by Theorem \ref{t0.3}, $V$ becomes a uniformly bounded
irreducible module for $W_n (R)$, for some finite dimensional commutative associative unital algebra $R$. Now,
$V$ is a uniformly bounded $W_n$-module. Hence its $A_n$ cover $\widehat{V} :=  \mathrm{span}\{\phi(x,v):x \in W_n, v \in V ,\,\,\, \phi(x,v) \in \mathrm{Hom}(A_n,V)\}$ is a uniformly bounded $\LL$-module.
 Then, consider the vector space $\widehat V \otimes R$, to which an $\LL(R)$-module structure can be given as follows:
\begin{align*}
    D({\bb u,r})\otimes b.\phi(x,m) \otimes b' &= \phi([D({\bb u,r}),x],m) \otimes bb'+\phi(x,D({\bb u,r})\otimes b.m) \otimes b',\\
    t^{\bb r}\otimes b.\phi(x,m) \otimes b' &= \phi(t^{\bb r} x,m) \otimes bb',
\end{align*}
for all ${\bb u}\in \mathbb C^n,{\bb r} \in \mathbb Z^n, x \in W_n,m \in V, b,b' \in R.$  It is easy to see that this action
defines a module structure on $\widehat V \otimes R$. We also note the action of $A_n \otimes R$ is associative on $\widehat V \otimes R$.
Now observe that $\eta:\widehat V \otimes R \to V   $ defined by $\eta(\phi(x, m)b) = (x \otimes b).m$ is a 
surjective $W_n (R)$-module map. As $\widehat V$ is uniformly bounded and R is finite dimensional, the
module $\widehat V \otimes R$ is a uniformly bounded $\LL (R)$-module. Therefore there is a composition
series for the $\LL \otimes R$-module $\widehat V \otimes R.$  Consider the composition series 
$$ 0 = V_0 \subseteq V_1 \subseteq  V_2 \subseteq \cdots \subseteq V_s = \widehat V \otimes R, $$
where $V_i/V_{i-1}$ is an irreducible uniformly bounded $\LL \otimes R$-module for $1 \leq i \leq s$. Let $1 \leq k \leq s$
be the smallest positive integer such that $\eta(V_k) \neq 0$. Then $\eta : V_k/V_{k-1} \to V$ is
a surjective $W_n (R)$-module map. Therefore $V$ is isomorphic to an irreducible
quotient of $V_k/V_{k-1}$, hence a single point evaluation module by Theorem 4.5 of \cite{serpc}.
\end{proof}
Note that all the proofs of our results in this section go through if we assume $n= 1$.

\subsection{Final result: Uniformly bounded case}
We assimilate the information from Proposition \ref{prom1}, Theorem \ref{mthmub1}, and Theorem \ref{mthmub2}  in the following
\begin{thm} \label{mthmub}
Let $n \geq 1$, then any non-trivial irreducible uniformly bounded $\LL_{S,n}(B)$-module can be considered as an irreducible $\LL(B)$ or $W_n (B)$-module. Furthermore, as a $\LL(B)$ or $W_n (B)$-module V is a single point evaluation module.
\end{thm}

\begin{remark}
For n = 1, Theorem \ref{mthmub} was proved in \cite[Theorem 5.3]{wrc} for more general Lie algebras $\LL_{S, 1}$, where $S$ is a Lie superalgebra with an even diagonalisable derivation.
In particular, Theorem \ref{mthmub} generalises the result of \cite{wrc} when $S$ is assumed to be a finite dimensional abelian Lie algebra.
\end{remark}
In the next section, we will deal with the irreducible $\LL_{S,n}(B)$-modules with finite but non-uniformly bounded weight spaces. We will see that the classification of those modules hinges on the classification result
of irreducible uniformly bounded $\LL_{S,n}$-modules. In particular, the following result will be used repeatedly in the next section:
\begin{thm}
Let $W$ be a non-zero $\LL_{S,n}$-module ($n \geq 1$) with uniformly bounded weight spaces. Then, $W$ contains a non-zero irreducible $\LL_{S,n}$-module. Furthermore,
\begin{enumerate}[label= \((\arabic*)\)]
\item If $ {\bb 0} \notin \mathrm{Supp}(W)$, then $W$ is an infinite dimensional $\LL_{S,n}$-module and contains a  non-trivial  irreducible $\LL_{S,n}$-submodule.
\item If $ {\bb 0} \in \mathrm{Supp}(W)$, then either $W$ is finite dimensional trivial $\LL_{S,n}$-module or $W$ is an infinite dimensional $\LL_{S,n}$-module and $W/W_1$ contains a 
  non-trivial irreducible $\LL_{S,n}$-module, where $W_1$ is the maximal finite dimensional trivial $\LL_{S,n}$-submodule of $W$.
\end{enumerate}
\end{thm}
\begin{proof}
The existence of a non-zero irreducible $\LL_{S,n}$-module is guaranteed by the fact that $W$ has a composition series of finite length. The results (1) and (2) follow by applications of the Theorem \ref{prom 4}, Theorem \ref{prom 5}, Theorem \ref{prom1}, Theorem \ref{prom 2}, and Theorem \ref{prom 3} respectively.
\end{proof}

\section{Non-uniformly bounded $\mathfrak{L}_{S, n}(B)$-modules}
In this section we classify all the irreducible Harish-Chandra $\mathfrak{L}_{S, n}(B)$-modules with non-uniformly bounded weight spaces. We need to introduce the following 
\subsection{Highest weight modules} We start with recalling some result proved for $\mathrm{Vir}$ in \cite{sav}. Let $\mathrm{Vir}^- \oplus \mathrm{Vir}^0 \oplus \mathrm{Vir}^+$ be standard triangular decomposition of $\mathrm{Vir}$, where
$\mathrm{Vir}^-  = \mathrm{span}\{x_i \mid i \in \Z_{<0}\}$, $\mathrm{Vir}^0 = \mathrm{span} \{C, x_0\}$ and $\mathrm{Vir}^+  = \mathrm{span}\{x_i \mid i \in \N\}$.
Let 
$\mathrm{Vir}(B)^{-} \oplus \mathrm{Vir}(B)^{0} \oplus \mathrm{Vir}(B)^{+}$ be a triangular decomposition of $\mathrm{Vir}(B)$, where
$\mathrm{Vir}(B)^{-} = \mathrm{Vir}^{-} \otimes B$, $\mathrm{Vir}(B)^{0} = \mathrm{Vir}^{0} \otimes B$ and $\mathrm{Vir}(B)^{+} = \mathrm{Vir}^{+} \otimes B.$
\begin{definition}
A $\mathrm{Vir}(B)$-module $V$ is called a highest weight module if it is a weight module and there exists a non-zero weight vector $v$ such that $\mathrm{Vir}(B)^+ . v = 0$ and $U(\mathrm{Vir})(B) = V$.
\end{definition}
Let $\C_{\psi}$ be the one dimensional representation of $\mathrm{Vir}(B)^{0}$, where $\psi \in \mathrm{Hom}(\mathrm{Vir}(B)^{0}, \C)$. Make $\C_{\psi}$ as a $\mathrm{Vir}(B)^{0} \oplus \mathrm{Vir}(B)^{+}$ with a trivial action of
$\mathrm{Vir}(B)^{+} $on it. Then consider Verma module $$M(\psi) := \mathrm{Ind}_{\mathrm{Vir}(B)^{0} \oplus \mathrm{Vir}(B)^{+}}^{U(\mathrm{Vir}(B))} \C_{\psi}.$$
Let $V(\psi)$ be the unique irreducible quotient. It is proved in \cite{sav} (Proposition 5.1) that $V(\psi)$ is Harish-Chandra module iff there exist a cofinite ideal $I$ of $B$ such that $\mathrm{Vir} \otimes I. V(\psi) = 0$.
The following result also proved in \cite{sav} (Theorem 5.3)
\begin{thm}
Every irreducible Harish-Chandra $\mathrm{Vir}(B)$-module with non-uniformly bounded weight spaces is a highest weight module and is a tensor product of finitely many irreducible single point generalised evaluation highest weight modules.
\end{thm}
One  can define highest weight modules for $\LL_{S,1}$ as above. We have similar result holds for $\LL_{S,1}(B)$:
\begin{thm}
Every irreducible Harish-Chandra $\LL_{S,1}(B)$-module with non-uniformly bounded weight spaces is a highest weight module and is a tensor product of finitely many irreducible single point generalised evaluation highest weight modules.
\end{thm}
\begin{proof}
Follows from Lemma 2.7 of \cite{wrc} and by similar techniques that of proof of Theorem 5.3 in \cite{sav}.
\end{proof}
We will now assume $n \geq 2$.
Let  $\mathfrak{L}_{S,n} = \bigoplus_{{\bb m} \in \Z^n}{{(\mathfrak{L}_{S,n})}_{{\bb m}}}$ be the root space decomposition of $\mathfrak{L}_{S, n}$ with respect to $H$, where
\begin{equation}
 {(\mathfrak{L}_{S,n})}_{\bb m} =
    \begin{cases}
      H \ltimes (S \otimes  t^{{\bb 0}}) & \text{if ${\bb m = 0}$}\\
       \big(\bigoplus_{i =1}^n{\C t^{\bb m} d_i} \big) \ltimes (S \otimes  t^{\bb m})  & \text{otherwise}.
      
    \end{cases}        
\end{equation}
Let $M$ be a subgroup of $\Z^n$ and ${\bm \beta} \in \Z^n$ such that $\Z^n = M \oplus \Z {\bm \beta}$. Consider a triangular decomposition of ${ \mathfrak{L}^-_{S,n}} \oplus {(\mathfrak{L}_{S,n})}_{M} \oplus {\mathfrak{L}^+_{S,n}}$ of $\LL_{S,n}$, where $${\mathfrak{L}^-_{S,n}} = \bigoplus_{d \in \N}
 {(\mathfrak{L}_{S,n})}_{M - d{\bm \beta}}, {\mathfrak{L}^+_{S,n}} = \bigoplus_{d \in \N} {(\mathfrak{L}_{S,n})}_{M + d{\bm \beta}}, {(\mathfrak{L}_{S,n})}_{M} = \bigoplus_{{\bb r} \in M}{{(\mathfrak{L}_{S,n})}_{\bb r}}.$$
 \begin{definition} 
 An $\LL_{S,n}$-module $V$ is called a highest weight module if it is a weight module and there exists a non-zero weight vector $v$ with the property that ${\mathfrak{L}_{S,n}}^+ . v = 0$ and $U(\LL_{S,n}).v = V$.
 \end{definition}
 \noindent
 Let $Y$ be an irreducible module for ${(\mathfrak{L}_{S,n})}_{M}$.  Make $Y$ as ${(\mathfrak{L}_{S,n})}_{M}  \oplus {\mathfrak{L}^+_{S,n}}$-module with a trivial action of ${\mathfrak{L}_{S,n}}^{+}$ on it. The generalised Verma module for $\LL_{S,n}$ is defined as 
 $$V_{\LL_{S,n}}(Y, {{\bm \beta}}, M) := \mathrm{Ind}_{({\LL_{S,n})}_{M} \oplus {\LL^+_{S,n}}}^{U(\LL_{S,n})} Y.$$
 Now by the standard arguments $V_{\LL_{S,n}}(Y, {\bm \beta}, M)$ contains a unique irreducible sub-quotient $L_{\LL_{S,n}}(Y, {\bm \beta}, M)$. By the results of Zhao and Billig \cite{ybkz}, it follows that $L_{\LL_{S, n}}(Y, {\bm \beta}, M)$ has finite dimensional weight spaces if $Y$ is a uniformly bounded 
 exp-polynomial module for ${({\LL}_{S,n})}_M$ (see \cite{ybkz} for definition). We have the following result:
 \begin{thm} \label{hwbc}
 Let $V$ be an irreducible $\LL_{S,n}$-module with finite but not uniformly bounded weight spaces. Then $V \cong L_{\LL_{S,n}}(Y, {\bm \beta}, M)$ for some $Y, {\bm \beta}, M$ defined as above.
 \end{thm}
 \begin{proof} Proof follows by exactly similar argument as that of proof of Theorem 4.3 in \cite{kgx}.
 \end{proof}
 Similarly, let $\mathfrak{L}_{S,n}(B) = \bigoplus_{{\bb m} \in \Z^n}{{\mathfrak{L}_{S,n}(B)}_{\bb m}}$  be the root space decomposition of $\mathfrak{L}_{S,n}(B)$ with respect to $H$, where 
\begin{equation}
 {\mathfrak{L}_{S, n}(B)}_{\bb m} =
    \begin{cases}
      (H \ltimes (S \otimes t^{{\bb 0}}))\otimes B & \text{if ${\bb m = 0}$}\\
     \big( \big( \bigoplus_{i =1}^n{ \C t^{\bb m} d_i} \big) \ltimes S \otimes t^{\bb m} \big) \otimes B & \text{otherwise}.
      
    \end{cases}       
\end{equation}
As above we define the generalised Verma module $$V_{\LL_{S,n}(B)}(Y, {\bm \beta}, M) : =  \mathrm{Ind}_{{\LL_{S,n}(B)}_M \oplus {\LL_{S,n}(B)}^{+}}^{U({\LL_{S,n}(B)})} Y$$ and $L_{\LL_{S,n}(B)}(Y, {\bm \beta}, M)$ be its unique irreducible quotient. Again by \cite{ybkz}, it follows that 
$L_{\LL_{S,n}(B)}(Y, {\bm \beta}, M)$ has finite dimensional weight spaces if $Y$ is a uniformly bounded ${\LL_{S,n}(B)}_M$-module.

\noindent{{\bb{Automorphisms of $\LL_{S,n}$}:}}

\noindent Let $GL_n (\Z)$ be the group of invertible matrix with entries in $\Z$. Let $A = (a_{ij}) \in GL_n (\Z)$ with $\mathrm{det}(A) = \pm1$. Then $A$ defines an Lie algebra automorphism 
$\Phi_{A} : \LL_{S,n} \rightarrow \LL_{S,n}$ as
$$\Phi_{A} (t^{{\bb m}} d_j) = \sum_{p =1}^{n}{b_{jp} t^{{\bb m}A^{T}} d_p}, \,\,\, \Phi_{A}(s \otimes t^{{\bb m}}) = s \otimes t^{{\bb m} A^{T}}, $$
where $1 \leq j \leq n, {\bb m} \in \Z^n, s \in S$ and $A^{-1} = B = (b_{i j})$.  The term up to a change of variable will always mean twisting of an $\LL_{S,n}$ action on its module by a suitable automorphism $\Phi_{A}$.

We need to recall the definition of a generalised highest weight vector. Let ${\bm \beta}, {\bm \gamma} \in \Z^n$. We denote ${\bm \beta} \geq {\bm \gamma}$ if  $\beta_i \geq \gamma_i$ for all $1 \leq i \leq n$ and ${\bm \beta} > {\bm \gamma}$ if  $\beta_i > \gamma_i$ for all $1 \leq i \leq n$.
\begin{definition}
Let $M$ be any $\LL_{S,n}$-module. A vector $v \in M$ is called a generalised highest weight vector if there exists $N \in \N$ such that $(\LL_{S,n})_{\bm \alpha}.v = 0$ for all ${\bm \alpha} > (N, \ldots, N)$.
\end{definition}
\subsection{A key result}
Let ${\bm \lambda} \in \C^n$ and let $W$ be any weight module for $\LL_{S,n}(B)$ or $\LL_{S,n}$.  Let $K$ be any subgroup of $\Z^n$, then we define $W[{\bm \lambda} + K] := \bigoplus_{{\bb k} \in K}{W_{{\bm \lambda}+{\bb k}}}$. In particular we have
$\LL_{S,n}(B) [K] = \bigoplus_{{\bb k} \in K}{(\LL_{S,n} (B))_{{\bb k}}}.$
The following theorem plays a key role in the classification:	
\begin{theorem}\label{thmnu}
	Let $V$ be an irreducible Harish-Chandra $\LL_{S,n}(B)$-module with non-uniformly bounded weight spaces.  Then $V$ is an $\mathfrak{L}_{S,n} (B)$-highest weight module.
\end{theorem}
\begin{proof}
It is enough to prove that there exists a non-zero weight vector $z$ with weight $\bm \lambda$ and an integer $p$ with the property that
$\LL_{S, n}(B)[{(p + j}) \bm \beta + K] .z = 0$ for all $j \geq 1$, where $\Z^n = \Z \bm \beta \oplus K$. Then by irreducibilty of $V$ and PBW theorem, it will imply that the $\mathrm{Supp}(V) \subseteq \{\bm \lambda + (p - j)\bm \beta + K: j \in \Z_{+} \}$. Hence we will have
$V = U(\LL_{S, n} (B)).Z$, where  $Z = U(\LL_{S, n} (B) [K]).z = V[{\bm \lambda + p \bm \beta + K}]$ is the highest weight space of $V$.

Now, since $V$ is irreducible, there exists ${\bm \mu} \in \C^n$ such that Supp $V \subseteq {\bm \mu} + \Z^n$. Choose a  ${\bm \gamma}\in \C^n$ with the property that $(\bm{{\bm \gamma}}|\bf{m}) \neq 0$ for all $\bf{m}\in \Z^n \setminus \{0\}$. Then consider $V$ as $\mathrm{Vir}(\bm{\bm \gamma})$-module and by exactly similar argument as that of Theorem 4.3 of \cite{kgx} we get $\bf{m}\in\Z^n$ such that $m_1=0$ and  
\begin{equation}
\mathrm{dim} V_{{\bm \mu}-\bf{m}}>(n+ \mathrm{dim}\, S)(\mathrm{dim}V_{{{\bm \mu}+\bf{e}}_1}+\sum_{i=2}^{n}\mathrm{dim}V_{{\bm \mu}+{\bf{e}}_1+{\bf{e}}_i})
\end{equation}	

Let ${\bf{q}}_1={\bf{m}}+{\bf{e}}_1$, ${\bf{q}}_i={\bf{m}}+{\bf{e}}_1+{\bf{e}}_i$ for $ i=2,\dots,n$. For simplification, let $\lambda:= {{\bm \mu}}-\bf{m}$. Then by (4.3) we have non-zero vector $v \in V_\lambda$ such that $D({\bf{e}}_i, {\bf{q}}_j). v= 0$ and $S\otimes t^{{\bf{q}}_j}.v=0$ for all $i,j=1,\dots,n$. It follows that $v$ is an $\mathfrak{L}_{S,n}$-generalised highest weight vector. Now consider the $\mathfrak{L}_{S,n}$-submodule $W:=U(\mathfrak{L}_{S,n}).v$ of $V$.  Let $\overset{\circ}{W}$ be the submodule of $W$ containing trivial vectors \footnote{We will prove at the end of this section (Proposition
\ref{propl}) that $\overset{\circ}{W} = \{{\bb 0}\}.$}, i.e., $\overset{\circ}{W} = \{v \in W \mid \LL_{S,n} . v = 0\}$. 
Consider the quotient module $\overline{W}:= W/ \overset{\circ}{W}$.

\noindent
{\bf{Claim 1}:} An $\LL_{S,n}$-module $\overline{W}$ contains no non-zero trivial vectors. 

\noindent
Suppose contrary, there exists a trivial vector $\bar{w} = w + \overset{\circ}{W}$, with $w \notin \overset{\circ}{W}$ such that $\LL_{S,n}w \in \overset{\circ}{W}$. But as $\bar{w}$ is a trivial vector, it must lie in the zero weight space of $\overline{W}$. It implies that the weight of $w$ is also a zero 
weight vector considered as a weight vector of $W$.
But we have $(\LL_{S,n})_{\pm {\bb e_i}}.w \in \overset{\circ}{W}$, and as all vectors $\overset{\circ}{W}$ have weight zero (considered as a weight vectors of W), it follows that $(\LL_{S,n})_{\pm {\bb e_i}}. w = 0 $. Hence we have $\LL_{S,n}.w = 0$, which is absurd. This proves claim 1. Let $\bar{v}$ denote the image of $v$ under the natural map from $W$ to $\overline{W}$. Then we have $\overline{W} = 
U( \LL_{S,n}). (\bar{v})$.

\noindent
{\bf{Claim 2:}}  $\overline{W}$ contains a highest weight $\LL_{S,n}$-module.

\noindent
By claim 1, $\overline{W}$ contains no non-zero trivial vectors. We have the following:
 \begin{enumerate}[label=(\alph*)]
         \item Every vector in $\overline{W}$ is a generalised highest weight vector.
	 \item For any $\bm{\alpha}\in \N^n$ and any $0\neq \bar{w} \in \overline{W}$, $ (\mathfrak{L}_{S,n})_{-\bm{\alpha}}.\bar{w} \neq 0$.
	
	\item  For $\eta\in$ Supp$(\overline{W})$, and any $\bm{\alpha}\in \N^n$, the set $\{\ell\in \Z| \eta+\ell\bm\alpha\in$ Supp$(\overline{W})\}=(-\infty, s]$ for some $s\in\Z_{+}$. 
	\item Upto a change of variables  we have $\lambda+\bm\alpha\notin$ Supp$(\overline{W})$ for any non-zero $\bm\alpha\in\Z_{+}^n$.
\end{enumerate}

Proof of above statements are exactly similar to Lemma 3.3, 3.4, 3.5, 3.6 of \cite{vmkz} respectively. By (d) it follows that $\overline{W}$ is not a uniformly bounded $\mathfrak{L}_{S,n}$-module. So let $\mathfrak{L}_{S,n}(K,\bm\beta,X)$ be its irreducible sub-quotient. Now as every element of $\overline{W}$ is a generalised highest weight vector, replacing $\bar{v}$ by a suitable vector if necessary we can assume that 
\begin{center}
	$\bm \lambda-\Z_{+}\bm\beta+K \setminus \{{\bb 0}\} \subseteq $ Supp$(\overline{W})$.
\end{center}  

We also have the following properties (see Lemma 3.7, \cite{vmkz})
\begin{enumerate}[label=(\roman*)]
	\item $\bm \lambda +\bm\alpha+K \nsubseteq$ Supp($\overline{W}$) for all ${\bm \alpha} \in \Z_{+}^n \setminus \{{\bb 0}\}$.
	\item  $\exists~ {\bm{\nu}}\in \N^n$ such that $K:= \{{\bf{y}}\in\Z^n : ({\bm{\nu}|\bf{y})}= 0\}$.
	\item $\bm \lambda+\bf{y}\in$ Supp($\overline{W}$) $\forall~ {\bf{y}}\in\Z^n$ with $(\bm\nu|\bf{y})<0$.
\end{enumerate}

\noindent
{\bf{Subclaim}:}
	$\{\bm \lambda +k\bm\beta+ K: k\in\N\}\bigcap$ Supp$(\overline{W})\neq \phi$ for only finitely many $k\in\N$.

\noindent
	Assume contrary, so we get a strictly increasing sequence of $k_i\in\N$ such that 
	\begin{center}
			$\{\bm \lambda +k_i\bm\beta+K: i\in\N\}\bigcap$ Supp$(\overline{W})\neq \phi$.
	\end{center}

Now by choosing sufficiently large $k_i\in \N$ and taking a subsequence of $\{k_i\in\N\}$ if necessary we can write
\begin{equation} \label{e1}
|\{\bm \lambda +k_r\bm\beta+K\}\bigcap \{\bm \lambda +\Z_{+}^n\}|>1
\end{equation}
\begin{equation}
|\{\bm \lambda +k_{r-1}\bm\beta+K\}\bigcap \{\bm \lambda +\Z_{+}^n\}|>0
\end{equation}

\noindent
By equation \ref{e1}, we have that $W[\bm \lambda +k_{r}\bm\beta+K]$ is a non-uniformly bounded $\mathfrak{L}_{S,n}[K]$-module. Now following the argument of the proof of Lemma 3.7 line by line we arrive at a contradiction. This completes the proof of subclaim.
Let $j_0\in\Z_{+}$ be the largest integer such that 
\begin{equation} \label{impeq}
	\{\bm \lambda +j_0\bm\beta+K \}\bigcap  \mathrm{Supp}(\overline{W})\neq \phi.
\end{equation}
Consider $\overline{W}[\bm \lambda +j_0\bm\beta+K]$ as $\mathfrak{L}_{S,n}[K]$-module. Observe that $\overline{W}[\bm \lambda +j_0\bm\beta+K]$ is a uniformly bounded $\mathfrak{L}_{S,n}[K]$-module. Otherwise it would contain an irreducible $\mathfrak{L}_{S,n}[K]\bigcap \mathrm{Vir}(\bm{\bm \gamma})$ sub-quotient say $\bar{Y}$ which is not uniformly bounded. But then the unique irreducible quotient of $$\mathrm{Ind}_{U((\mathfrak{L}_{S,n}[K] \cap \mathrm{Vir}({\bm \gamma}))+({\mathfrak{L}_{S,n}}^+ \bigcap \mathrm{Vir}(\bm{\bm \gamma})))}^{U(\mathrm{Vir}(\bm{\bm \gamma}))}\bar{Y}$$ is not a Harish-Chandra module which will imply that $\overline{W}$ is not a Harish-Chandra module.
\noindent
 Now as $\overline{W}[\bm \lambda +j_0\bm\beta+K]$ is uniformly bounded $\mathfrak{L}_{S,n}[K]$-module, let $\overline{Z}$ be its non-zero irreducible submodule. It is easy to see that the module $\bar{M} := U(\LL_{S,n}).\overline{Z}$ is the required highest weight $\mathfrak{L}_{S,n}$-module
 contained in $\overline{W}$. This completes the proof of claim 2. Now we need to deal with two cases:
 
 \noindent
 {\bf{Case 1:}} ${\bb 0} \notin \mathrm{Supp}({\overline{W}})$. Note that in this case all the zero weight vectors of $W$ lie in $\overset{\circ}{W}$. Let $\overline{Z}$ be its non-trivial irreducible submodule of $\overline{W}[\bm \lambda +j_0\bm\beta+K]$.
 Let $0 \neq \bar{z} \in \overline{Z}$  be any weight vector. Let $z$ be any preimage of $\bar{z}$ under the surjective $\LL_{S,n}[K]$-map $\Phi: W[{\bm \lambda}+j_0{\bm\beta}+K] \rightarrow \overline{W}[{\bm \lambda}+j_0{\bm\beta}+K]$. Let us define $M := U(\LL_{S,n}).z$. Let $M'$ be the maximal
 $\LL_{S,n}$-submodule of $M$ with property that $M_{{\bm \gamma}}' = 0$ for all $\bm \gamma \in \mathrm{Supp}(\overline{Z})$. The existence of $M'$ follows from the fact that $M$ is a cyclic module and equation \ref{impeq}. Observe also that $M \cap \overset{\circ}{W} \subseteq M'$.
 Let us consider the restriction of map $\Phi$ on an $\LL_{S,n}[K]$-module $M \{K\} := U(\LL_{S,n} [K]).z$, which we denote by $\psi$. Then $\psi: M \{K\} \mapsto \overline{Z}$ is a surjective $\LL_{S,n} [K]$-module map defined by $\psi(u.z) = u.\bar{z}$, where $u \in \LL_{S,n} [K]$. We have
 $\mathrm{Kernel}(\psi) = M\{K\} \cap \overset{\circ}{W}$.
 
  \noindent
 {\bf{Claim 3:}} $M/M'$ is an irreducible $\LL_{S,n}$-module.
 
 \noindent
  Assume on the contrary that there exists a proper non-zero $\LL_{S, n}$-submodule $M_1$ of $M$ containing $M'$ such that $(M_1)_{{\bm \gamma}} \neq 0$ where ${\bm \gamma} \in \mathrm{Supp} (\overline{Z})$.
 Let $0 \neq z_1 \in (M_1)_{{\bm \gamma}}$. Then there exists $u \in U(\LL_{S,n})$ such that $u.z = z_1$. 
 
\noindent 
 {\bf{Subclaim:}} $u \in U(\LL_{S,n}[K])$.
 
 \noindent
 Let us assume that the subclaim is not true. Then $u$ must be of the form $u_1 u_2 u_3$, where $u_1 \in U_0(\LL_{S,n}^-)$, $u_2 \in U(\LL_{S,n}[K])$ and $u_3 \in U_0(\LL_{S,n}^+)$. Here, for any Lie algebra $\mathfrak{g}$, $U_0(\mathfrak{g})$ denotes the argumentation ideal of
 $U(\mathfrak{g})$. But by equation \ref{impeq}, it follows that $u_3.z = 0$ unless weight of the vector $u_3.z$ is zero. But as all the zero weight vectors in M lies in $\overset{\circ}{W}$, it follows that $u_3 .z \in \overset{\circ}{W}$ and hence $z_1 = 0$ which is a required contradiction.
 This proves the subclaim. Now, as $\psi(z_1) = \bar{z_1}$ is a non-zero element of $\overline{Z}$ and $\overline{Z}$ is an irreducible $\LL_{S,n}[K]$-module, there exists an element $u_1 \in U(\LL_{S,n}[K])$ such that $u_1 . \bar{z_1} = \bar{z}$. But this implies that $z -  u_1.z_1 \in M \cap \overset{\circ}{W} 
 \subseteq M'$. So, we have $z + M' = u_1.z_1 + M'$, which is a contradiction to the assumption that $M_1/M'$ is a proper submodule of $M/M'$. This completes the proof of claim 3.
 
 By Theorem \ref{hwbc}, we have $M/M' \cong L_{\LL_{S,n}}(K, {\bm \beta}, Z)$,
 where $Z = U({\LL_{S, n}} [K]).\bar{z}'$, where $\bar{z}'$ is image of $z$ under the natural map from $M$ to $M/M'$.
 Let ${\bf{g}}_2\in K$ such that $V_{\bm \lambda +j_0 {\bm\beta}+{\bf{g}}_2}\neq 0$. As $L_{\mathfrak{L}_{S,n}}(K,{\bm \beta}, Z)$ is non-uniformly bounded there exists ${\bf{g}}_1\in K$ and $p\in\N$ such that 
\begin{equation}
\mathrm{dim}{(L_{\mathfrak{L}_{S,n}}(K,{\bm \beta}, Z))}_{\bm \lambda +(j_0-p)\bm\beta+{\bf{g}}_1} )> (\mathrm{dim}\,S+n)\mathrm{dim} V_{\bm \lambda +j_0\bm\beta+{\bf{g}}_2}
\end{equation}
Let $0\neq f\in B$. Let $ M''_{\bm \lambda +(j_0-p)\bm\beta+{\bf{g}}_1}$ be a vector space compliment of $M'_{\bm \lambda+(j_0-p)\bm\beta+{\bf{g}}_1}$ in $M_{\bm \lambda+(j_0-p)\bm\beta+{\bf{g}}_1}$. Then we have

\begin{center}
 $\mathrm{dim}M''_{\bm \lambda+(j_0-p)\bm\beta+{\bf{g}}_1}>(\mathrm{dim}\,\,S+n)\mathrm{dim}V_{\bm \lambda+j_0\bm\beta+{\bf{g}}_2}$.
\end{center}
By above inequality there exists $0\neq w_{f,p}^{{\bf{g}}_1}\in M_{\bm \lambda+(j_0-p)\bm\beta+{\bf{g}}_1}$ such that
\begin{center} ${((\mathfrak{L}_{S,n})_{p\bm\beta+{\bf{g}}'}\otimes f)w_{f,p}^{{\bf{g}}_1}=0}$, where $ {\bf{g}}' = {\bb g_2 - g_1}$.
	\end{center} 
Let us choose $N \in \N$ be large enough so that $\bm \lambda + (p + j) \bm \beta + K \neq \bb 0$ for all $j \geq N$.	
	Then we have
\begin{center}
	${(\mathfrak{L}_{S, n})_{(p+j)
		\bm\beta+{\bf{g}}}.w_{f,p}^{{\bf{g}}_1}=0}$ $\forall j\geq N$ and $\forall {\bf{g}}\in K$.
\end{center}
As $L_{\mathfrak{L}_{S,n}}(K,{\bm \beta}, Z)$ is irreducible, there exists $u_1,\dots,u_\ell$ and $h_1,\dots,h_\ell$ such that 
\begin{center}
	$h_1\cdots h_qu_1\cdots u_\ell.w_{f,p}^{{\bf{g}}_1} = \bar{z}',$
\end{center}
where
 $u_i\in(\mathfrak{L}_{S, n})_{(j_i)
	 {\bm\beta}+{\bf{g}}_i''}, 1\leq i \leq \ell, j_i>0, \sum_{i = 1}^{\ell}j_i = p, \,\,  \mathrm{and}\,\, {\bf{g}}_i''\in K
 \,\,  \mathrm{and}\,\, h_k\in(\mathfrak{L}_{S, n})_{{\bf{g}}_k'},$ $  {\bb g}_k'\in K,\,\, \forall \,\, 1\leq k\leq q.$
 But as $M_{\bm \gamma}' = \{0\}$ for all ${\bm \gamma} \in \mathrm{Supp} (\overline{Z})$, we have $h_1\cdots h_q u_1\cdots u_\ell.w_{f,p}^{{\bf{g}}_1} = z.$
Now, using 
	 ${((\mathfrak{L}_{S, n})_{p\bm\beta+{\bf{g}}'}\otimes f)w_{f,p}^{{\bf{g}}_1}=0}	={(\mathfrak{L}_{S, n})_{(p+j)\bm\beta+{\bf{g}}}.w_{f,p}^{{\bf{g}}_1}},$ for all $j \geq N$,
 we have
 \begin{equation}
 {((\mathfrak{L}_{S, n})_{(2p+j)\bm\beta+{\bf{g}}}\otimes f)w_{f,p}^{{\bf{g}}_1}} = [{(\mathfrak{L}_{S, n})_{(p+j)\bm\beta+{\bf{g-g}}_1}, {((\mathfrak{L}_{S, n})_{p\bm\beta+{\bf{g}}_1}\otimes f)}].w_{f,p}^{{\bf{g}}_1}}=0 \, \forall {\bb g} \in K.
 \end{equation}
 But then for all $j\geq N$,	 
 \begin{center}
 	$((\mathfrak{L}_{S, n})_{(2p+j)\bm\beta+{\bf{g}}}\otimes f)z=((\mathfrak{L}_{S, n})_{(2p+j)\bm\beta+{\bf{g}}}\otimes f)h_1\cdots h_qu_1\cdots u_\ell.w_{f,p}^{{\bf{g}}_1}=0$ $\forall {\bb g}\in K $
  \end{center}
 by an induction on $q + \ell$ and using $((\mathfrak{L}_{S, n})_{(2p+j)\bm\beta+{\bf{g}}}\otimes f).w_{f,p}^{{\bf{g}}_1}=0$ $\forall {\bb g}\in K, j\geq N.$
 
 \noindent
 {\bf { Case 2 :}} ${\bb 0} \in \mathrm{Supp}({\overline{W}})$. In this case we have two subcases: 
 
 \noindent
 {\bf {Subcase (i):}} $\{\bm \lambda +j_0\bm\beta+K\} \cap \mathrm{Supp} (\overline{W}) \neq \{ \bb 0\}$.
 Let $ \overset{\circ}{\overline{W}}[ \bm \lambda + j_0 \bm \beta + K]$ be a maximal trivial $\LL_{S,n}[K]$-module of
 $\overline{W}[\bm \lambda + j_0 \bm \beta + K]$. It follows that $$ \overset{\circ}{\overline{W}}[ \bm \lambda + j_0 \bm \beta + K] = \overset{\circ}{{W}}[K] \cap W[\bm \lambda + j_0 \bm \beta + K]/ \overset{\circ}{W} \cap W[\bm \lambda + j_0 \bm \beta + K],$$
 where $\overset{\circ}{W}[K] : = \{w \in W :  \LL_{S,n}[K].w = 0\}$. Now, $\overline{W}[\bm \lambda + j_0 \bm \beta + K]/ \overset{\circ}{\overline{W}}[\bm \lambda + j_0 \bm \beta + K]$ contains a  non-trivial irreducible $\LL_{S,n}[K]$-submodule, which we denote by $\overline{D}$.
 We have $$W[\bm \lambda + j_0 \bm \beta + K]/ \overset{\circ}{W}[K] \cap W [\bm \lambda + j_0 \bm \beta + K] \cong \overline{W}[\bm \lambda + j_0 \bm \beta + K]/ \overset{\circ}{\overline{W}}[\bm \lambda + j_0 \bm \beta + K] $$ as $\LL_{S,n}[K]$-modules.
 Consider a natural surjective $\LL_{S,n}[K]$-module map $$\Phi: W[\bm \lambda + j_0 \bm \beta + K] \rightarrow W[\bm \lambda + j_0 \bm \beta + K]/\overset{\circ}{W}[K] \cap W [\bm \lambda + j_0 \bm \beta + K]. $$ Let $0 \neq \bar{s} \in \overline{D}$ be a non-zero vector of non-zero weight. Let $s$ be any preimage of
 $\bar{s}$ under $\Phi$. Let $T = U(\LL_{S, n}).s$ and $T'$ be its maximal $\LL_{S,n}$-module with the property that $T'_{\bm \eta}  = 0$ for all $\bm \eta \in \mathrm{Supp}(\overline{D}) \setminus \{0\}$. It is easy to see that $T \cap \overset{\circ}{W}[K] \subseteq T'$. Let us define an $\LL_{S,n}[K]$-module $T\{K\}: = U(\LL_{S,n}[K]).s$ and an $\LL_{S,n}[K]$-module surjective map $\psi: T\{K\} \mapsto \overline{D}$ by $\psi(u.s) = u.\bar{s}$, where $u \in U(\LL_{S,n}[K]).$ Note that $\psi$ is a restriction of the map $\Phi$ on $T\{K\}$ and we have 
 $T \{K\}/ T\{K\}\cap \overset{\circ}{W}[K]  \cong \overline{D}.$
 
 \noindent
 {\bf{Claim 4:}} $T/T'$ is an irreducible $\LL_{S,n}$-module.
 
 \noindent Let us assume that the claim is false. So there exist a proper $\LL_{S,n}$-submodule $T_1$ containing $T'$ such that $(T_1)_{\bm \gamma} \neq 0$ for some $\bm \gamma \in \mathrm{Supp}(\overline{D}) \setminus \{\bb 0\}$. Let $0 \neq s_1 \in (T_1)_{\bm \gamma}.$ Then there
 exists $u \in U(\LL_{S,n})$ such that $u.s = s_1$. We see here that $u$ must lie in $U(\LL_{S,n}[K])$; otherwise $u = u_1 u_2 u_3$ where $u_1 \in U_0(\LL_{S,n}^-)$, $u_2 \in U(\LL_{S,n}[K])$ and $u_3 \in U_0(\LL_{S,n}^+)$. But we notice that
 $u_3. s = 0$ unless weight of $u_3. s$ is zero. But this possibility is ruled out by the choice of $j_0$ in the equation \ref{impeq}. Hence we have $u \in U(\LL_{S,n}[K])$. Now, consider the image $\psi(s_1) = \bar{s_1}$, which is non-zero as $s_1 \notin T'$. But then irreducibility of
 $\overline{D}$ implies that there exists an $u_1 \in U(\LL_{S,n}[K])$ such that $u_1.(\bar{s_1}) = \bar{s}$. But this forces that $s + T' = u_1.s_1 + T'$, which is a contradiction to the assumption that $T_1/T'$ is a proper submodule of $T/T'$ and we are done with the proof of claim 4.
 Now, we have $T/T' \cong L_{\LL_{S,n}}(K,  {\bm \beta}, D)$, where $D = U(\LL_{S,n}[K]) \bar{s}'$, where $ \bar{s}'$ is the image of $s$ in $T/T'$. Now, rest of the proof follows exactly along the similar lines as that of proof of case 1.
 
 \noindent
 {\bf {Subcase (ii):}} $\{\bm \lambda +j_0\bm\beta+K\} \bigcap \mathrm{Supp} (\overline{W}) = \{ \bb 0\}$. In this subcase notice that $\overline{W}[\bm \lambda +j_0 \bm \beta+K] = \overset{\circ}{W}[K]/ \overset{\circ}{W} = \overline{W}_{\bb 0}$ ,where
 $\overset{\circ}{W}[K] : = \{w \in W : \LL_{S,n}[K].w = 0\}$  Now,  let $0 \neq \bar{w} = w + \overset{\circ}{W}$. Let $j_1$ be the largest integer with the property that $j_1 < j_0$ and $\{\bm \lambda +j_1\bm\beta+K\} \bigcap \mathrm{Supp} (\overline{W}) \neq \phi$.
 But using the property of the Lie algebra $\LL_{S,n}$ that $[(\LL_{S,n})_{\bb m}, (\LL_{S,n})_{\bb s}] = (\LL_{S,n})_{\bb m + s}$ for ${\bb m, s} \in \Z^n$ and ${\bb m} \neq {\bb s}$, it follows that $j_1 = j_0 -1$. For if $\{ \bm \lambda + (j_0 - 1) \bm \beta+ K\} \bigcap \mathrm{Supp} (\overline{W}) = \phi$,
 then it would imply $\LL_{S,n}[- \bm \beta + K].w = 0$. But then using the before mentioned bracket property of $\LL_{S,n}$, it will follow that $\LL_{S,n}^- .w = 0$. As we already have $\LL_{S,n}^+ .w = 0 = \LL_{S,n}[K].w$, it follows that $\LL_{S,n}.w = 0$, which is a contradiction to the assumption
 that $w \notin \overset{\circ}{W}$. Let $x \in (\LL_{S,n})_{-\bm \beta + k}$ such that $x.w \neq 0$, for some $k \in K$. Now using the facts that $\LL_{S,n}^+ .w = 0 = \LL_{S,n}[K].w$, a quick calculation shows that $\LL_{S,n}^+ . (x.w) = 0$. Let $z_1 := x.w$ and consider the $\LL_{S,n}[K]$-module
 $Z_1 = U(\LL_{S,n}[K]).z_1$. An easy check shows that $\LL^+_{S,n}. y = 0$  for all $y \in Z_1$. Observe that $Z_1$ must be a uniformly bounded $\LL_{S,n}[K]$-module. Otherwise, considering it as $\mathrm{Vir}(\bm \gamma) \cap \LL_{S,n}[K]$-module and taking its irreducible subquotient $\overline{Z_1}$, which has finite dimensional but non-uniformly bounded weight spaces, we get a unique irreducible  quotient of $\mathrm{Vir}(\bm \gamma)$-module 
 $$\mathrm{Ind}_{U((\mathfrak{L}_{S,n}[K] \cap \mathrm{Vir}({\bm \gamma}))+({\mathfrak{L}_{S,n}}^+ \bigcap \mathrm{Vir}(\bm{\bm \gamma})))}^{U(\mathrm{Vir}(\bm{\bm \gamma}))}\bar{Z_1},$$
 which does not have finite dimensional weight spaces. This will eventually imply that $W$ does not have finite dimensional weight spaces.  Now, as $Z_1$ is uniformly bounded and ${\bb 0} \notin \mathrm{Supp}(Z_1)$, it contains a non-trivial irreducible $U(\LL_{S,n}[K])$-submodule
 say $Z$. Then $M : = U(\LL_{S,n}).Z$ is a highest weight $\LL_{S,n}$-module contained in $W$. Let $M/M'$ be its unique non-trivial irreducible quotient. Now, we can proof that any $0 \neq z \in Z$ is an  $\LL_{S,n}(B)$ highest weight vector by using the same argument as in case 1.   
 \end{proof}
 
 Let $L_{\LL_{S,n}(B)}(K,  {\bm \beta}, Z)$ be an  irreducible highest weight  $\LL_{S,n}(B)$-module with finite dimensional weight spaces. The highest weight space $Z$ is an irreducible $(\LL_{S,n}(B))_K$-module with uniformly bounded weight spaces. We aim to show that 
 $L_{\LL_{S,n}(B)}(K,  {\bm \beta}, Z)$ is a single point evaluation $\LL_{S,n}(B)$-module.  We start with the following:
 \begin{prop} \label{propimp} The $(\LL_{S,n}(B))_K$-module $Z$ is a single point evaluation module.
 \end{prop}
 \begin{proof}
In the view of Theorem \ref{mthmub}, it is enough to prove that $(\LL_{S,n}(B))_K$ is isomorphic to $(\LL_{S',n-1}(B))$ as a Lie algebra, where $S'$ is some finite dimensional abelian Lie algebra. It is enough to prove that 
$(\LL_{S,n}(B))_K$ is isomorphic to the Lie algebra $\mathcal{G}: = \mathrm{span}\{D({\bb u , r}) \otimes b, s \otimes t^{{\bb s}} \otimes b' \mid {\bb r, s} \in \Z^n, r_n = 0  = s_n, s \in S, b,b' \in B\}$ as $\mathcal{G}$ is isomorphic
to $\LL_{S', n-1}(B)$, with $S' := S \oplus \C {\bb e_n}$ by the assignments
$D({\bb{u, r}}) \otimes b \mapsto D({\bb{u, r}}) \otimes b$ for ${\bb u} \in \C^n$ and ${\bb r} \in \Z^n$ with $u_n = 0 = r_n$,   $D({\bb{e_n, r}}) \otimes b \mapsto {\bb e_n} \otimes t^{{\bb r}} \otimes b$, and $s \otimes t^{{\bb s}} \otimes b' \mapsto s \otimes t^{\bb s} \otimes b'.$
Let ${\bm {\alpha_1}, \ldots, {\bm \alpha_{n-1}}}$ be elements of $\Z^n$ such that $K = \sum_{i = 1}^{n-1}{\Z {\bm \alpha_i}}$, and let ${\bm \alpha_n} \in \Z^n$ such that $\Z^n = K \oplus \Z {\bm \alpha_n}$. Let $\{\bm{\gamma_1}, \ldots ,\bm{\gamma_n}\}$ be the basis of
$\mathbb{R}^n$ which is dual to $\{{\bm {\alpha_1}, \ldots, {\bm \alpha_{n}}}\}$. Then the map $\Phi : \mathcal{G} \rightarrow (\LL_{S,n}(B))_K$ by
$\Phi (t^{\bb{r}}d_i \otimes b) = D({\bm \gamma_i}, \sum_{i = 1}{r_i {\bm \alpha_i}})$  and $\Phi(s \otimes t^{{\bb{r}}} \otimes b) = s\otimes t^{\sum_{i = 1}{r_i {\bm \alpha_i}}} \otimes b$ for $1 \leq i \leq n$ and $r_n = 0$. An easy checking shows that the map $\Phi$ is a Lie algebra isomorphism
and we are done.
 \end{proof}

\begin{theorem} \label{prom2}
$L_{\LL_{S,n}(B)}(K,  {\bm \beta}, Z)$ is a single point evaluation $\LL_{S,n}(B)$-module. 
\end{theorem}
\begin{proof}
By Proposition \ref{propimp}, $Z$ is an evaluation $(\LL_{S,n}(B))_K$-module. So there is an algebra homomorphism $\psi: B \rightarrow \C$ such that 	
$$(D({\bb u,r})\otimes b).z=\psi (b)D({\bb u, r}).z,  (s\otimes t^{\bb m}\otimes b'). z= \psi(b')(s\otimes t^{\bb m}).z,$$
where ${\bb u} \in \C^n, {\bb r,m} \in K$ and $b, b' \in B, s \in S$ and $z \in Z$.  Also as ${L_{S, n} (B)}^+$ acts trivially on $Z$, we have
$$(D({\bb u}, p_1 {\bb \beta} + {\bb r})\otimes b).z = 0 =  (s\otimes t^{l_1 {\bm \beta} + \bb m}\otimes b').z,$$
for all $b, b' \in B$ and $p_1$ and $l_1$ are any positive integers. It implies that  $$(D({\bb u}, p_1 {\bb \beta} + {\bb r})\otimes(b - \psi(b))z = 0 = (s\otimes t^{l_1 {\bm \beta} + \bb m})\otimes (b'- \psi(b'))z,$$
where ${\bb u} \in \C^n, {\bb r,m} \in K$ and $b, b' \in B, s \in S, p_1, l_1 \in \N$ .

\noindent
{\bf Claim :} For ${\bm \alpha}\in \mathbb{Z}^n, {\bm \alpha} = p {\bm\beta} + {\bb r} $ with $ p \leq -1$ and ${\bb r} \in K$, 
$(D({\bb u}, {\bm \alpha})\otimes (b-\psi(b))).z =0= ((S\otimes t^{\bm {\alpha}})\otimes (b-\psi(b))).z$ for all $b\in B$.

\noindent
We will prove the claim by proving that $(D({\bb u}, {\bm \alpha})\otimes (b-\psi(b))).z$ is a highest weight vector of $V$ which does not lie in highest weight space, hence must be zero. We will use induction on $p$.
Let $p=-1.$ Let ${\bb v}\in \mathbb{C}^n, {\bm {\bm \gamma}} \in \mathbb{Z}^n, {{\bm {\bm \gamma}}} = l {\bm \beta} + {\bb s}$ where $ l \in \N, {\bb s} \in K $ and $b'\in B $ . We have
$$(D({\bb v, {\bm {\bm \gamma}}})\otimes b')(D({\bb u}, { \bm \alpha})\otimes(b -\psi (b))).z = (D({\bb w}, {\bm \alpha} + {\bm {\bm \gamma}})\otimes(b'b-b' \psi(b))).z,$$
where ${\bb w}=({\bb v}, { \bm \alpha} ){\bb u}-({\bb u}, { \bm {\bm \gamma}}){\bb v}.$
\noindent	
Now, if $ p + l \geq 1$ then $(D({\bb w}, {\bm \alpha} + {\bm {\bm \gamma}})\otimes(b'b-b'\psi(b))).z = 0$. In the case $p + l =0 $. We have 
$$(D({\bb w}, {\bm \alpha} + {\bm {\bm \gamma}})\otimes(b'b-b'\psi(b))).z = \psi(b'b-b'\psi(b))(D({\bb w}, {\bm \alpha} + {\bm {\bm \gamma}}).z =0.$$
Similarly, we have 
		$$(S\otimes t^{\bm {\bm \gamma}}\otimes b') (D({\bb u}, { \bm \alpha})\otimes(b -\psi (b))).z = 0.$$ Now, it is easy to prove using induction that the above identity holds for all $p \leq -1$.
It follows that $(D({\bb u, {\bm \alpha}})\otimes (b-\psi(b))).z$ is a highest weight vector, hence must be zero.
Similarly we  also have $((S\otimes t^{\bm \alpha})\otimes (b-\psi(b))).z =0 $ for all ${\bm  \alpha} \in \mathbb{Z}^n$ with $ p \leq -1$, $b\in B$.
	
	 Consider the set 
	\begin{center}
		$T : = \{v\in V~|~ (D({\bb u,r})\otimes (b-\psi(b))).v =0=((S\otimes t^{\bb m})\otimes (b-\psi(b))).v ~~ \forall b\in B, {\bb r,m}\in \mathbb{Z}^n \}.$
	\end{center}
	By above it follows that $T \neq 0$ as $Z \subseteq T$. It is an easy checking that $T$ is $\mathfrak{L}_{S,n}(B)$-submodule of $V$, hence must be $V$.
\end{proof}
In the light of the above theorem we concentrate on $L_{\LL_{S,n}}(K, {\bm{\beta}}, Z)$, an irreducible highest weight $\LL_{S,n}$-module with finite dimensional weight spaces. We have the following.

\begin{theorem}
	For an  $\mathfrak{L}_{S,n}$-module $L_{\LL_{S,n}}(K, {\bm \beta}, Z)$, there exists a co-dimension one subspace $\overset{\circ}{S}$ of $S$ such that $\overset{\circ}{S} \otimes A_n$ acts trivially on $L_{\LL_{S,n}}(K, {\bm{\beta}}, Z)$. In particular, 
\begin{enumerate}[label= \((\arabic*)\)]
\item If $S \otimes t^{\bb 0}$ acts nontrivially, then $L_{\LL_{S,n}}(K, {\bm{\beta}}, Z)$ is an irreducible highest weight module for $\LL$.
\item If $S \otimes t^{\bb 0}$ acts trivially, then $L_{\LL_{S,n}}(K, {\bm{\beta}}, Z)$ is an irreducible highest weight module for $W_n$.
\end{enumerate}
\end{theorem}
\begin{proof} By application of Proposition \ref{prom1} and Theorem \ref{prom2}, there exists a co-dimension 1 subspace $\overset{\circ}{S}$ such that $ \overset{\circ}{S} \otimes t^{\bb m}$ acts  trivially on $Z$ for all ${\bb m} \in K$. We also have 
	
		$$\overset{\circ}{S} \otimes t^{\bm \alpha}.Z=0$$ for all ${\bm \alpha}\in \mathbb{Z}^n$ such that ${\bm \alpha} = p {\bm \beta} + {\bm s}, p \in \N, \bm s \in K $. 
\noindent
Following the same method as that of proof of Theorem \ref{prom2} we get
	
		$$\overset{\circ}{S} \otimes t^{\bm {\bm \gamma}}.Z=0$$ for all ${\bm {\bm \gamma}}\in \mathbb{Z}^n,   {\bm {\bm \gamma}} = \ell {\bm \beta} + {\bb s}, \ell \in \Z_{< 0}, {\bb s} \in K$
and the set $$\{v\in L_{\LL_{S,n}}(K, {\bm \beta}, Z)~ | \,\,\,\, \overset{\circ}{S} \otimes t^{\bb r}.v=0 ~\forall {\bb r}\in \mathbb{Z}^n \}$$ is a non-zero submodule of $L_{\LL_{S,n}}(K, {\bm \beta}, Z)$, hence must be $L_{\LL_{S,n}}(K, {\bm \beta}, Z)$ itself. This completes proof of 
$(1)$. Proof of $(2)$ follows by arguments similar to that of $(1)$.
\end{proof}
\subsection{Final result: Non-uniformly bounded case}
We sum up the results of this section in the following:
\begin{thm} \label{mthmnub}
Let $V$ be an irreducible Harish-Chandra $\LL_{S,n}$-module with non-uniformly bounded weight spaces, then $V$ can be considered as an irreducible $\LL(B)$ or $W_n (B)$-module. Further, considered as an $\LL(B)$ or $W_n (B)$-module, V
 is a single point evaluation module and upto a change of variables $V \cong L_{\LL}(K, {\bm \beta}, Z)$ or $ L_{W_n}(K, {\bm \beta}, Z)$.
 \end{thm}
\subsection {Applications and remarks}
\noindent
We start with following:
\begin{definition}
	Let $V$ be any $\mathfrak{L} (B)$-module. An element $v\in V$ is called $\mathfrak{L}$-trivial if $\mathfrak{L}.v = 0$.
\end{definition}
\begin{prop} \label{propl}
Let $V$ be a non-trivial irreducible $\LL(B)$-module with finite dimensional weight spaces. Then $V$ does not contain any non-zero $\LL$-trivial vector.
\end{prop}
\begin{proof}
Let us begin on the contrary that there exists $0 \neq w \in V$ such that $\LL.w = 0$. Then there must exist ${\bb u}_1 \in \C^n, {\bb r}_1 \in \Z^n$ and $b \in B$ such that $D({\bb u_1, \bb r}_1)\otimes b. w \neq 0$.
Let $W := \mathrm{span}\{D({\bb u, \bb r}) \otimes b.w : \forall\,\, {\bf u} \in \C^n, {\bb r} \in \Z^n\}$. It is trivial checking that $W$ acquires a $W_n$-module structure and the map
$\Psi : \mathrm{ad}\,W_n \rightarrow W$ defined by $\Psi(D({\bb u, \bb r})) = D({\bb u, \bb r})\otimes b.w$ is a $W_n$-module map. Now, as $\Psi$ is a non-zero map, the irreducibility of $\mathrm{ad}\, W_n$ implies
that $\Psi$ is an isomorphism. But then we have $\mathrm{Supp} (V) = {\bm \gamma} + \Z^n,$ where ${\bm \gamma}$ is the weight of $w$. An immediate application of Theorem \ref{thmnu} implies that $V$ must be a uniformly bounded
$\LL(B)$-module. But then appealing to Theorem \ref{mthmub} we get that $D({\bb u, \bb r}) \otimes b'.w = \eta(b')D({\bb u, \bb r}).w = 0$ for all ${\bb u} \in \C^n, {\bb r} \in \Z^n, b' \in B$ and $\eta$ is the algebra homomorphism from $B$ to $\C$. In particular we have $D({\bb u, \bb r}) \otimes b.w=0$ for all
${\bb u} \in \C^n, {\bb r} \in \Z^n$, which is a required contradiction.
\end{proof}
We have the following result which generalises the result by Savage \cite{sav}:
\begin{cor} Every irreducible Harish-chandra module for $W_n (B)$ $(n \geq 2)$ is either uniformly bounded or a highest weight module. Further, every irreducible Harish-Chandra $W_n (B)$-module is a single point evaluation module.
\end{cor}
\begin{proof}
Taking the abelian Lie algebra $S$ to be zero, the result is now immediate from Theorem \ref{mthmub} and Theorem \ref{mthmnub}.
\end{proof}
\begin{remark}.
We would like to mention the significant difference between classification result of irreducible non-uniformly Harish-Chandra module for $n =1$ and $n \geq 2$. Unlike the case $n = 1$ where  irreducible non-uniformly bounded
Harish-Chandra modules are tenor product of finitely many irreducible single point generalised evaluation highest weight modules, for $n \geq 2$, every irreducible Harish-Chandra module (uniformly bounded or non-uniformly bounded) is a single point evaluation module.

\end{remark}

\noindent{\bf {Acknowledgments:}} Authors would like to thank Rencai Lu for some helpful discussions. First author would also like to thank Sudipta Mukherjee for some helpful discussions.

\end{document}